\documentclass[10pt]{article}

\usepackage[a4paper]{geometry}
\usepackage{amsthm}
\usepackage{amsmath,amsfonts,amssymb}
\usepackage{graphicx,verbatim,mathtools,color}
\usepackage[only,llparenthesis,rrparenthesis]{stmaryrd}
\usepackage{bm}
\usepackage[textsize=tiny,color=blue!30!white,obeyFinal]{todonotes}

\usepackage{adjustbox}
\usepackage{subcaption}
\DeclareCaptionSubType*[arabic]{figure}

\definecolor{dblue}{rgb}{0.,0.,0.8}

\usepackage{hyperref}
\hypersetup{colorlinks, colorlinks, linkcolor=dblue, citecolor=dblue,urlcolor=dblue, plainpages=false, pdfwindowui=false,pdfstartview={FitH}}
\usepackage{hypcap}

\usepackage{tikz,setspace}
\usetikzlibrary{patterns}
\usetikzlibrary{positioning}
\usetikzlibrary{external}
\usepackage{pgfplots}

\usepackage{multirow}
\usepackage{rotating}

\numberwithin{equation}{section}

\theoremstyle{plain}
\newtheorem{theorem}{Theorem}[section]
\newtheorem{lemma}[theorem]{Lemma}
\newtheorem{corollary}[theorem]{Corollary}
\theoremstyle{remark}
\newtheorem{remark}{Remark}[section]
\newtheorem{example}{Example}[section]

\newcommand{\bo}[1]{\mathbf{#1}} %% BOLD
\newcommand{\lld}{\llparenthesis}
\newcommand{\rrd}{\rrparenthesis}

\newcommand{\calA}{\mathcal{A}}

\newcommand{\calD}{\mathcal{D}}
\newcommand{\calE}{\mathcal{E}}
\newcommand{\calI}{\mathcal{I}}

\newcommand{\calQ}{\mathcal{Q}}

\newcommand{\calP}{\mathcal{P}}
\newcommand{\calR}{\mathcal{R}}

\newcommand{\calM}{\mathcal{M}}
\newcommand{\calN}{\mathcal{N}}

\newcommand{\bD}{\bm{\calD}}

\newcommand{\bQ}{\widetilde{\bo Q}}

\newcommand{\eps}{\varepsilon}
\newcommand{\Id}{\mathrm{Id}}
\newcommand{\dd}{\mathrm{d}}

\newcommand{\abs}[1]{\lvert#1\rvert} 
\newcommand{\tends}{\rightarrow}
\newcommand{\norm}[1]{\lVert#1\rVert}
\newcommand{\p}{\partial}

\newcommand{\rhoa}{\rho_{\bo A}}
\newcommand{\rhouz}{\rho_{U}}
\newcommand{\lmax}{\lambda_{\mathrm{max}}}
\newcommand{\lmin}{\lambda_{\mathrm{min}}}
\newcommand{\gmin}{\gamma}
\newcommand{\gmax}{\Gamma}

\DeclareMathOperator{\Dim}{dim}

\DeclareMathOperator{\Diag}{diag}

\newcommand{\pair}[2]{\langle #1,#2 \rangle}
\newcommand{\R}{\mathbb{R}}
\newcommand{\Om}{\Omega}

\newcommand{\vphi}{\varphi}

\newcommand{\Vh}{\mathbb{V}}
\newcommand{\Vt}{\mathbb{V}_{\tau}}
\newcommand{\X}{\mathbb{A}}
\newcommand{\Y}{\mathbb{S}}

\newcommand{\Cadd}{C_{\Vh}^{\mathrm{add}}}

\newcommand{\Cmult}{C_{\Vh}^{\mathrm{mult}}}
\newcommand{\Csol}{C_{\Vh}^{\mathrm{prec}}}

\title{Time-parallel iterative solvers for parabolic evolution equations\thanks{This project has received funding from the European Research Council (ERC) under the European Union's Horizon 2020 research and innovation program (grant agreement No 647134 GATIPOR).}}
\author{Martin Neum\"uller\footnotemark[2] $\,$and Iain~Smears\footnotemark[3]}

\begin{document}

\renewcommand{\thefootnote}{\fnsymbol{footnote}}
\footnotetext[2]{Institute of Computational Mathematics, Johannes Kepler University Linz, 4040 Linz, Austria (neumueller@numa.uni-linz.ac.at)}
\footnotetext[3]{Department of Mathematics, University College London, 25 Gordon Street, London WC1E 6BT, United Kingdom (i.smears@ucl.ac.uk)}
\renewcommand{\thefootnote}{\arabic{footnote}}

\maketitle

\begin{abstract}
We present original time-parallel algorithms for the solution of the implicit Euler discretization of general linear parabolic evolution equations with time-dependent self-adjoint spatial operators.
Motivated by the inf-sup theory of parabolic problems, we
show that the standard nonsymmetric time-global system can be equivalently reformulated as an original symmetric saddle-point system that remains inf-sup stable with respect to the same natural parabolic norms.
We then propose and analyse an efficient and readily implementable parallel-in-time preconditioner to be used with an inexact Uzawa method.
The proposed preconditioner is non-intrusive and easy to implement in practice, and also features the key theoretical advantages of robust spectral bounds, leading to convergence rates that are independent of the number of time-steps, final time, or spatial mesh sizes, and also a theoretical parallel complexity that grows only logarithmically with respect to the number of time-steps.
Numerical experiments with large-scale parallel computations show the effectiveness of the method, along with its good weak and strong scaling properties.
\end{abstract}
\medskip

\smallskip

\tableofcontents

\section{Introduction}\label{sec:intro}
The usual approach to the numerical solution of parabolic partial differential equations (PDE) involves the sequential solution of discrete systems obtained by a time-stepping scheme. 
In many cases, a large number of time-steps might be required, which can lead to long computation times.
The large number of cores in present-day high-performance computers enables the use of time-parallel algorithms as an alternative to the sequential approach. By treating many time-steps simultaneously and in parallel, these specialized algorithms aim to compute the numerical solution to a desired accuracy in a shorter computational time than can be achieved by the sequential approach.

Starting with Nievergelt~\cite{Nievergelt1964} in 1964, a wide variety of time-parallel algorithms have been proposed~\cite{BotchevVanderVorst2001,ChristliebHaynesOng2012,EmmettMinion2012,Falgoutetal2014,GanderGuttel2013,GiladiKeller2002,HoangJaffreJaphet2013,HortonVandewalleWorley1995,McDonaldWathen2016,Womble1990}, see also the review~\cite{Gander2015}.
One of the central questions that any time-parallel algorithm must address is how to efficiently propagate information about the solution across the whole time interval.
In the parareal method~\cite{Bal2005,GanderVandewalle2007,LionsMadayTurinici2001,MadayTurinici2005}, this is done by iterating between the sequential solution of problems on a coarse temporal grid with parallel solves on the fine temporal grid.
Instead of a single coarse temporal grid, time and space-time multigrid methods~\cite{Falgoutetal2014,NeumullerGander2016,Hackbusch1984,Horton1992,HortonVandewalle1995} use a hierarchy of coarser grids in both space and time. 
For linear problems, many of these algorithms can be seen as convergent iterative solvers for a large time-global nonsymmetric linear system. Recall that the available theory for the design and analysis of iterative methods for general nonsymmetric systems is currently rather more limited than for their symmetric counterparts \cite{MalekStrakos2015,Wathen2015}. Therefore, the study of iterative methods for parabolic problems appears significant within the wider context of the solution of large nonsymmetric linear systems.

Recently, two independent works~\cite{Andreev2016,Smears2016} proposed to treat the nonsymmetry of linear systems coming from parabolic problems using approaches based on inf-sup theory.
In particular, first step in the approach from~\cite{Smears2016}, which is developed further here, is to build a left-preconditioner that is based on the mapping from trial functions to their optimal test functions in the analysis of the inf-sup condition. 
The left-preconditioner then leads to equivalent symmetric reformulations that are stable in the same norms and spaces (or their discrete analogues) as those that appear naturally in the analysis of well-posedness of the problem.
Note that the stability of these reformulations distinguishes this approach from classical ones such as forming the normal equations.
Convergent iterative solvers can then be built out of preconditioners for the symmetric positive definite systems associated to the norms and spaces appearing in the inf-sup theory.

This approach to preconditioning and solving the system is conceptually distinct from those in many previous works in several ways, one of which is that it explicitly features preconditioners that handle the appropriate spatial dual-norm on the time derivative.
For illustration, consider momentarily the heat equation $\p_t u - \Delta u = f$ in $\Om\times(0,T)$, where $f\in L^2(0,T;H^{-1}(\Om))$, with, for simplicity, homogeneous Cauchy--Dirichlet conditions at the boundary and initial time. 
Then the solution space $S$ is the space of functions in $L^2(0,T;H^1_0(\Om))\cap H^1(0,T;H^{-1}(\Om))$ that vanish at $t=0$ (see e.g.\ \cite{Wloka1987}).
Furthermore, we have the following identity
\begin{equation}\label{eq:continuous_inf_sup}
\begin{aligned}
\norm{u}_S = \sup_{v\in L^2(0,T;H^1_0(\Om))\setminus\{0\}} \frac{B(u,v)}{\norm{v}_A} &&&\forall\,u\in S,
\end{aligned}
\end{equation}
where $B(u,v)\coloneqq \int_{0}^T \pair{\p_t u}{v} + (\nabla u,\nabla v)_{\Om}\,\dd t$ is the bilinear form for a weak formulation of the problem, where $\norm{u}_S^2 \coloneqq \int_0^T \norm{\p_t u }_{H^{-1}(\Om)}^2 + \norm{\nabla u}_{\Omega}^2 \,\dd t + \norm{u(T)}_\Om^2$ is the norm on $S$,  and where $\norm{v}_A^2 \coloneqq  \int_0^T \norm{\nabla v}_\Omega^2 \,\dd t$, with $\norm{\cdot}_{\Omega}$ denoting the $L^2$-norm over $\Omega$ (for proof, see for instance~\cite{ErnSmearsVohralik2016}).
For each trial function $u\in S$, there is an optimal choice of test function $v\in L^2(0,T;H^1_0(\Om))$ that achieves the supremum in~\eqref{eq:continuous_inf_sup}; in the discrete setting, this mapping between trial and optimal test functions leads to a left-preconditioner that symmetrizes the system in a stable way. The complete solution algorithm combines this with additional preconditioners tied to the discrete versions of the norms $\norm{\cdot}_S$ and $\norm{\cdot}_A$, i.e.\ that handle explicitly the norm on the time derivative (see Sections~\ref{sec:formulation} and \ref{sec:schur_prec} below for details).

We note that the inf-sup theory for parabolic problems has previously found application in other contexts, such as a priori error analysis \cite{TantardiniVeeser2016}, a posteriori analysis~\cite{ErnSmearsVohralik2016}, and reduced-basis methods~\cite{UrbanPatera2014}. We also refer the reader to the textbooks~\cite{ErnGuermond2004,Schwab1998} for an  introduction to the inf-sup theorem for general linear equations in Banach spaces, and its application to parabolic problems. 

In this work, we present some original time-parallel algorithms for the solution of the implicit Euler discretization of general parabolic evolution equations with self-adjoint spatial operators.
The first contribution is to show that the discrete systems admit a similar inf-sup analysis to \eqref{eq:continuous_inf_sup}, which allows us to find equivalent symmetric systems that are stable with respect to the discrete counterparts of the norms $\norm{\cdot}_S$ and $\norm{\cdot}_A$ above.
In particular, we obtain an equivalent symmetric saddle-point formulation that is well-suited for preconditioned iterative solvers, such as the inexact Uzawa method \cite{BramblePasciakVassilev1997,Zulehner2002} or the preconditioned MINRES method~\cite{PaigeSaunders1975}.
The second contribution is to propose robust and efficient time-parallel preconditioners for these linear systems, resulting fast convergence of the iterative solvers.
Specifically, the preconditioner for the first variable of the saddle-point system is block-diagonal with respect to the time-steps, and the Schur complement preconditioner for the second variable is also block-diagonalized under the Discrete Sine Transform (DST) in time. 
Thus, after application of the DST, the time-parallelism is essentially trivial, and significantly simplifies the treatment of the dual norm of the time derivatives.
The DST can be implemented through parallel Fast Fourier Transforms (FFT), which have a low parallel complexity and are often relatively cheap in practice compared to the spatial solvers. The transformations of the temporal basis via the DST constitutes the main mechanism for exchange of information over time.

The main features of the algorithm can be summarized as follows.

\emph{Convergence theory.}
Applying the proposed time-parallel preconditioners to standard solvers, such as the inexact Uzawa method, leads to robust convergence rates that depend only on the efficiency of the spatial preconditioners and on the quasi-uniformity of the problem. As a result, the convergence of the algorithm is independent of the number of time-steps, the spatial mesh size, and the final time. The main step in the analysis is a proof of the robust spectral equivalence of the Schur complement and its preconditioner through the DST.

\emph{Parallel complexity.}
Since we are primarily interested here in time-parallelism, we study the dependence of the parallel complexity on $N$ the number of time-steps.
We show that each iteration of the method has a parallel complexity of $O(\log N)$ when sufficiently many processors are available.
To put this result in context, note that the standard parareal method using the implicit Euler method in the coarse and fine solvers achieves at best a parallel complexity of order $O(\sqrt{N})$, as shown in \cite{BalMaday2002}.
This result should be considered in light of the logarithmic-order lower bounds on the optimal achievable parallel complexity, see~\cite{HortonVandewalleWorley1995,Worley1991}.

\emph{Treatment of time-dependent spatial operators and non-uniform time-steps.}
The spatial operators can be time-dependent and the time-step lengths may vary, under the quasi-uniformity condition~\eqref{eq:tauA_assumption} below. 
It is worth noting that the time-dependence of the operators precludes an analysis based on reducing the problem to the scalar ODE case through spatial eigenvector decompositions, which is a common approach in the literature on time-parallel algorithms.

\emph{Simplicity of implementation.}
On the practical side, the parallel implementation of the proposed method has the advantage of being non-intrusive with respect to the spatial solvers and preconditioners, so black-box existing spatial solvers can be re-used.
This means that significant spatial parallelism can be straightforwardly included, see for instance the numerical experiments below.
The parallelization in time requires only parallel implementations of the one-dimensional FFT, which are available in libraries such as FFTW3 \cite{FFTW05}.

This article is organized as follows. The discrete parabolic problem is presented in Section~\ref{sec:formulation}, where we propose an equivalent stable symmetric saddle-point formulation that is the starting point for our approach.
We then consider an inexact Uzawa method as a representative iterative solver in Section~\ref{sec:uzawa}, along with a convergence theorem that motivates the need for a spectrally equivalent Schur complement preconditioner.
Section~\ref{sec:schur_prec} then details the construction of parallelisable preconditioner along with the key spectral bounds.
This is followed by the bounds on parallel complexity in Section~\ref{sec:parallel_complexity}.
The analysis of the spectral bounds is taken up in Sections~\ref{sec:infsup} and \ref{sec:schur}. Finally, we present numerical experiments with large scale parallel computations in Section~\ref{sec:numexp}, before presenting our conclusions.

\section{Discrete parabolic problem}\label{sec:formulation}
For $T>0$, consider a partition of the time interval $(0,T)$ into disjoint time-step intervals $I_n =(t_{n-1},t_n)$, with $0=t_0 \leq t_{n-1} < t_n \leq t_N = T$ for each $1\leq n \leq N$. Let $\tau_n \coloneqq t_n - t_{n-1}$ denote the time-step lengths for each $n=1,\dots, N$. 
For a given finite dimensional space $\Vh$, let $M$ and $A_n$, $n=1,\dots,N$, be symmetric positive definite matrices on $\Vh$.

Consider the discretization of an abstract parabolic evolution equation by the implicit Euler method 
\begin{equation}\label{eq:algo_euler}
\begin{aligned}
M (u_n - u_{n-1} ) + \tau_n A_n u_n = \tau_n f_n, 
\end{aligned}
\end{equation}
where $u_n\in \Vh$ for each $n=1,\dots,N$. At the first time-step $n=1$, the term $u_0$ is replaced by some given initial datum $u_I\in \Vh$.
In applications to second-order parabolic PDEs, the matrices $M$ and $A_n$ typically represent the mass and stiffness matrices obtained by some spatial discretization method.

The matrices $M$ and $A_n$, $n=1,\dots,N$, induce the inner-products $(\cdot,\cdot)_M$ and $(\cdot,\cdot)_{A_n}$ and the norms $\norm{\cdot}_{M}$ and $\norm{\cdot}_{A_n}$ on $\Vh$.
To simplify the notation, we shall identify functions in $\Vh$ with their vector representations, so for instance we shall write $\norm{v}_M^2 = v^\top M \,v$ for all $v\in\Vh$.
We assume that there exists a symmetric positive definite matrix $A$ along with positive constants $\tau>0$ and $\alpha\geq 1$, such that
\begin{equation}\label{eq:tauA_assumption}
\begin{aligned}
\frac{1}{\alpha}\, \tau A \leq \tau_n A_n \leq \alpha\, \tau A &&& \forall\,n=1,\dots,N,
\end{aligned}
\end{equation}
where the inequalities are in the sense of the partial ordering of positive semi-definite symmetric matrices (i.e.\ $A\leq B$ if and only if $B-A$ is positive semi-definite).
In other words, we assume that the matrices $\tau_n A_n$, $n=1,\dots,N$, are uniformly spectrally equivalent to the matrix $\tau A$; this amounts to a non-degeneracy and quasi-uniformity assumption.
For instance, this assumption is guaranteed if the temporal grid is quasi-uniform and all matrices $A_n$ are spectrally equivalent to $A$; however~\eqref{eq:tauA_assumption} is a somewhat weaker assumption in general.
We stress that \eqref{eq:tauA_assumption} is not a CFL-type restriction on the time-step sizes. 
We also remark that the constant $\tau$ and matrix $A$ will be needed in the algorithm below, and, ideally, they should be chosen to make the constant $\alpha$ as close to $1$ as possible. In practice, some simple choices would be to select $\tau$ and $A$ among the time-steps $\{\tau_n\}$ and operators $A_n$ given by the problem, or to consider an average.

\subsection{Equivalent reformulations} First, we express~\eqref{eq:algo_euler} in time-global form by gathering the solution values into the vector $\bo u = [u_1,\dots, u_N] \in \Vh^N$, with $\Vh^N = \Vh\times \dots\times \Vh$, which leads to the nonsymmetric system
\begin{equation}\label{eq:algo_B_system}
\begin{aligned}
\bo B  \bo u = \bo f, &&& \bo B \coloneqq \bo K + \bo A,
\end{aligned}
\end{equation}
where $\bo A \coloneqq \Diag \{\tau_n A_n \}_{n=1}^N$ is the block-diagonal matrix with entries $\tau_n A_n$ along the diagonal, and where $\bo K \coloneqq K\otimes M$ with $\otimes$ is the Kronecker product, with $K\in \R^{N\times N}$ defined by
\begin{equation}\label{eq:algo_K_def}
K \coloneqq \left(
\begin{smallmatrix}
 1 &  & &  \\
 -1 & 1 & & \\
 & -1 &1 &   \\
 &  & & \ddots 
\end{smallmatrix}
\right).
\end{equation}
Furthermore, the right-hand side in \eqref{eq:algo_B_system} is given by $\bo f = [ \tau_1 f_1 + M u_I, \tau_2 f_2, \dots, \tau_N f_N]$.

The starting point is to consider two equivalent reformulations of \eqref{eq:algo_B_system}.
Define the matrix $\bo P \coloneqq \bo A^{-1} \bo K + \bo I$, with $\bo I = \Id_{N\times\Dim\Vh}$ the identity matrix of dimension $N\times \Dim\Vh$.
Then, we define the left-preconditioned matrix $\bo S \coloneqq \bo P^{\top} \bo B$.
An easy calculation shows that the system~\eqref{eq:algo_B_system} is equivalent to solving the left-preconditioned system
\begin{equation}\label{eq:algo_S_system}
\begin{aligned}
\bo S \bo u = \bo g, && \bo S = \bo K^\top \bo A^{-1} \bo K + \bo K + \bo K^\top  + \bo A, && \bo g \coloneqq \bo P^{\top} \bo f.
\end{aligned}
\end{equation}
It will be shown in Section~\ref{sec:infsup} below that the matrix $\bo S$ is symmetric and positive definite, and that the left-preconditioner $\bo P$ represents the optimal choice of test function for the inf-sup stability analysis of $\bo B$. 
Specifically, we will show in Theorems~\ref{thm:discrete_inf_sup} and~\ref{thm:symm} below that 
\begin{equation}\label{eq:infsup_matrix}
\begin{aligned}
\norm{\bo u}_{\bo S} = \sup_{\bo v\in \Vh^N\setminus \{0\} } \frac{\bo v^{\top} \bo B \bo u}{ \norm{\bo v}_{\bo A} }  &&& \forall\, \bo u \in \Vh^N,
\end{aligned}
\end{equation}
where the supremum is achieved by taking $\bo v = \bo P \bo u$, and where $\norm{\cdot}_{\bo A}$ and $\norm{\cdot}_{\bo S}$ are the norms associated with $\bo A$ and $\bo S$ defined above.
Furthermore, the matrix $\bo S$ represents the discrete analogue of the continuous norm~$\norm{\cdot}_S$ appearing in~\eqref{eq:continuous_inf_sup}, as will be seen from the Galerkin interpretation of the implicit Euler method.

It is possible, in theory, to apply standard iterative solvers, such as a preconditioned conjugate gradient (CG) method, to~\eqref{eq:algo_S_system}.
However, this would explicitly require the action of $\bo A^{-1}$ at each iteration, which can be expensive to compute.
To overcome this issue, we introduce the auxiliary variable $\bo p \in \Vh^N$ defined by the equation $ \bo A \bo p =  \bo K \bo u - \bo f$. 
Then, a simple calculation shows that the matrix $\bo S$ in~\eqref{eq:algo_S_system} is the Schur complement of the symmetric indefinite system
\begin{equation}\label{eq:mixed_formulation}
\begin{aligned}
\bm{\calA}\, \bm{u} = \bm{g},
&&&
\bm{\calA} \coloneqq 
\begin{bmatrix}
\bo A & - \bo K \vspace{1ex} \\ 
- \bo K^\top & - \left(\bo K + \bo K^{\top} + \bo A\right)
\end{bmatrix},
&&&
\bm{u} \coloneqq
\begin{bmatrix}
  \bo p  \vspace{1ex}\\  \bo u
\end{bmatrix},
&&&
\bm{g} \coloneqq
\begin{bmatrix}
  -\bo{f} \vspace{1ex}\\ -\bo f
\end{bmatrix}.
\end{aligned}
\end{equation}
The solution $\bo u \in \Vh^N$ of \eqref{eq:algo_B_system} is the second component of the solution $\bm{u}\in \Vh^N\times \Vh^N$ of \eqref{eq:mixed_formulation}, and we see that $\bo p = - \bo u$ from~\eqref{eq:algo_B_system}.
The matrix $\bm{\calA}$ is of saddle-point type, with dimension $2\times\Dim \Vh \times N$ and is block-sparse, since $\bo A$ is block-diagonal, $\bo K$ is block lower-triangular and bidiagonal, and $\bo K+\bo K^\top$ is block-tridiagonal.

We propose to compute the solution of the problem by applying preconditioned iterative solvers to the system~\eqref{eq:mixed_formulation}, such as the inexact Uzawa method.
The advantage of \eqref{eq:mixed_formulation} over \eqref{eq:algo_S_system} is that it allows for iterative solvers that merely approximate the action of $\bo A^{-1}$ with preconditioners.
The advantage of \eqref{eq:mixed_formulation} over the original formulation~\eqref{eq:algo_B_system} is that it is symmetric, whilst remaining inf-sup stable with respect to the same norms as \eqref{eq:algo_B_system}.
Indeed, it is possible to show, using for instance the bounds in~\cite{PestanaWathen2015}, that, for any $\bm{u}\in \Vh^N\times \Vh^N$,
\begin{equation}\label{eq:inf_sup_mixed}
\frac{1}{2}\left(\sqrt{5}-1\right) \norm{\bm{u}}_{*} \leq \sup_{\bm{v}\in\Vh^N\times\Vh^N\setminus\{0\}} \frac{ \bm{v}^\top \bm{\calA}\, \bm{u} }{\norm{\bm{v}}_{*}} \leq \frac{1}{2}\left(\sqrt{5}+1\right) \norm{\bm{u}}_*.
\end{equation}
where the norm $\norm{\cdot}_*$ on $\Vh^N\times \Vh^N$ is defined by
\begin{equation}\label{eq:mixed_norm}
\begin{aligned}
  \norm{ \bm{v} }_{*}^2 \coloneqq \norm{\bo q}_{\bo A}^2 + \norm{\bo v}_{\bo S}^2 &&& \forall\, \bm{v} = \left[ \bo q, \bo v \right]\in \Vh^{N}\times\Vh^N.
\end{aligned}
\end{equation}
We then see that the norm on the second variable in~\eqref{eq:mixed_norm} is the same as in the left-hand side of \eqref{eq:infsup_matrix}. 

\section{Inexact Uzawa method}\label{sec:uzawa}
There is a range of iterative methods for solving saddle-point systems such as \eqref{eq:mixed_formulation}. We consider here the inexact Uzawa method
\begin{equation}\label{eq:uzawa}
\begin{aligned}
\bo{p}_{j+1} &= \bo{p}_{j} +   \widetilde{\bo A}^{-1} \left(\bo K \bo{u}_{j} - \bo{A} \bo{p}_j - \bo f  \right), \\
\bo{u}_{j+1} &= \bo{u}_{j} + \omega \widetilde{\bo H}^{-1} \left( \bo f - \bo{K}^\top \bo{p}_{j+1} - \left[\bo K + \bo{K}^\top + \bo A \right] \bo{u}_j\right),  
\end{aligned}
\end{equation}
where $\widetilde{\bo A}$ and $\widetilde{\bo H}$ are respectively  preconditioners for $\bo A$ and $\bo S$, where $\omega>0$ is a damping parameter, and where $\bm{u}_0 = [\bo{p}_0, \bo{u}_0]$ is an initial guess.
In practice, it is most natural to choose $\widetilde{\bo A}$ to be of block-diagonal form
\begin{equation}\label{eq:Aapprox_block_form}
  \widetilde{\bo A} \coloneqq \Diag \{ \tau_n \widetilde{A}_n \}_{n=1}^N,
\end{equation}
where, for each $n=1,\dots,N$, the matrix $\widetilde{A}_n$ is symmetric positive definite. The application of $\widetilde{\bo A}^{-1}$ is then trivially parallel with respect to the time-steps.
In practice, the approximation~$\widetilde{A}_n$ is usually defined implicitly in terms of its inverse $\widetilde{A}_n^{-1}$ that represents the action of a standard solver for spatial problems, such as a small number of multigrid V-cycles for example. 
For the analysis, we assume that there exists $\rhoa$ with $0\leq \rhoa <1$, such that each $\widetilde{A}_n$ is a convergent approximation of $A_n$, i.e.\ 
\begin{equation}\label{eq:Aapprox_convergent}
\begin{aligned}
\norm{I - \widetilde{A}_n^{-1} A_n }_{\widetilde{A}_n} \leq \rhoa &&& \forall \, n=1,\dots, N.
\end{aligned}
\end{equation}

The following theorem gives sufficient conditions for convergence of the inexact Uzawa method. Let the norm $\norm{\cdot}_{\bD}$ be defined on $\Vh^N\times\Vh^N$ by
\begin{equation}\label{eq:Dnorm_def}
\norm{\bm{v}}_{\bD}^2 \coloneqq \omega \rhoa \norm{\bo q}_{\widetilde{\bo A}}^2 + \norm{\bo v}_{\widetilde{\bo H}}^2.  
\end{equation}

\begin{theorem}[Convergence]\label{thm:uzawa}
Suppose that \eqref{eq:Aapprox_convergent} holds and that $\widetilde{\bo H}$ is symmetric positive definite. Let $\lmax>0$ (respectively $\lmin>0$) be an upper bound on the maximum eigenvalue (respectively lower bound on the minimum eigenvalue) of $\widetilde{\bo H}^{-1} \bo S$.
Let the sequence of iterates $\{\bm u_j\}_{j=0}^{\infty}$ be defined by~\eqref{eq:uzawa}, and let the quantities $\sigma_{-}$ and $\sigma_{+}$ be defined by
\begin{subequations}\label{eq:sigma_pm_def}
\begin{align}
\sigma_{-} &\coloneqq \frac{1}{2} \left[ (1-\rhoa)( 1-\omega\lmin ) + \sqrt{ 4\rhoa  + (1-\rhoa)^2(1-\omega \lmin)^2 } \right], \label{eq:sigma_min}
\\
\sigma_{+} &\coloneqq  \frac{1}{2} \left[ 
(1+\rhoa)(1+\omega \lambda_{\max})-2 + \sqrt{ 4\rhoa + \left[(1+\rhoa)(1+\omega \lambda_{\max})-2 \right]^2 }  \right] . \label{eq:sigma_plus}
\end{align}
\end{subequations}
Let $\rhouz \coloneqq \max\{\sigma_{-},\sigma_{+}\}$. 
Then we have
\begin{equation}\label{eq:uzawa_contraction}
\begin{aligned}
\norm{\bm{u}-\bm{u}_{j+1}}_{\bD} \leq \rhouz \norm{ \bm{u}-\bm{u}_j }_{\bD}  &&& \forall\,j\geq 0.
\end{aligned}
\end{equation}
If the damping parameter $\omega>0$ is chosen such that
\begin{equation}\label{eq:damping_condition}
  \omega\, \lmax  < 2\,  \frac{1 - \rhoa}{1+\rhoa} ,
\end{equation}
then $\rhouz <1 $ and the inexact Uzawa method is convergent.
\end{theorem}

The proof of Theorem~\ref{thm:uzawa} is given in Section~\ref{sec:uzawa_proof} below.
Thus, it is seen that the convergence rate of the inexact Uzawa method can be bounded by $\rhouz$ which depends only on $\rhoa$, $\omega$ and on the spectral equivalence between $\widetilde{\bo H}$ and $\bo S$.
If $\widetilde{\bo H}$ is spectrally equivalent to $\bo S$, i.e.\ with uniform bounds on $\lmin$ and $\lmax$ independent of the discretization parameters, e.g.\ the number of time-steps, then the convergence will be robust with respect to these parameters.
We will construct in Section~\ref{sec:schur_prec} below a spectrally equivalent preconditioner $\widetilde{\bo H}$ that verifies the assumptions in Theorem~\ref{thm:uzawa}, provided only that \eqref{eq:tauA_assumption} and \eqref{eq:Aapprox_convergent} hold, and that suitable preconditioners for certain associated spatial matrices are available.

\begin{remark}[Condition on the damping parameter]\label{rem:damping}
The condition~\eqref{eq:damping_condition} is essentially equivalent to the condition given in \cite[Thm~4.3]{Zulehner2002}, although written in a different form. This assumption is also rather natural, since if $\rhoa$ approaches zero, then the inexact Uzawa method approaches a preconditioned Richardson iteration for $\bo S$ and~\eqref{eq:damping_condition} approaches the standard condition that $\omega < 2/\lmax$ for guaranteeing the contraction of the iteration matrix $\bo I-\omega \widetilde{\bo H}^{-1} \bo S$.
\end{remark}

\begin{remark}[Norm of the principal variable]\label{rem:max_norm}
Since we are primarily interested in the second variable in the system~\eqref{eq:mixed_formulation}, and since the preconditioners $\widetilde{\bo H}$ considered here will be shown to be spectrally equivalent to $\bo S$, we see that the associated norm for the second variable in $\norm{\cdot}_{\bm{\calD}}$ is equivalent to $\norm{\cdot}_{\bo S}$, which is natural for the problem, owing to~\eqref{eq:infsup_matrix}.
In the literature, various other norms have been used in the analysis of parallel algorithms for parabolic equations; a popular choice appears to be the max-norm $\max_{1\leq n \leq N}\norm{v_n}_{M}$, see for instance~\cite{GanderVandewalle2007}.
It turns out that the norm $\norm{\cdot}_{\bo S}$ considered here is stronger than the max-norm; indeed, Corollary~\ref{cor:max_norm_bound} below shows that, for every $\bo v \in \Vh^N$, with $\bo v=[v_1,\dots,v_N]$, we have the bound $\max_{1\leq n \leq N}\norm{ v_n }_M \leq \norm{\bo v}_{\bo S}$, without any unknown constant. However, in general, this bound has no robust converse, i.e. the norm $\norm{\cdot}_{\bo S}$ is strictly stronger than the max-norm. In this sense, Theorem~\ref{thm:uzawa} gives stronger guarantees than robust convergence in the max-norm.
\end{remark}

\begin{remark}[Norm of the auxiliary variable]
We see in~\eqref{eq:Dnorm_def} that the parameters $\rhoa$ and $\omega$ appear in association with the norm for the auxiliary variable $\bo q$ in the definition of the norm~$\norm{\cdot}_{\bm{\calD}}$. This can be explained as follows: if one of $\rhoa$ or $\omega$ is very small, then it is clear that the error $\bo u - \bo u_{j+1}$ for the principal variable is not significantly influenced by the previous error  $\bo p-\bo p_j$ in the auxiliary variable; this is reflected in the bound~\eqref{eq:uzawa_contraction}.
In the case where $\rhoa$ approaches zero, then the inexact Uzawa method approaches the preconditioned Richardson iteration for $\bo S$, and the error $\bo u - \bo u_{j+1}$  becomes in the limit independent of $\bo p-\bo p_j$. For this reason, we see that the presence of the parameters $\rhoa$ and $\omega$ in norm~$\norm{\cdot}_{\bm{\calD}}$ is rather natural for the analysis of the inexact Uzawa method.
\end{remark}

\begin{remark}[Other iterative methods]
Although we consider here the inexact Uzawa method for solving the system~\eqref{eq:mixed_formulation}, this is by no means the only possible choice. For instance, one alternative is the MINRES algorithm \cite{PaigeSaunders1975}. 
Given a symmetric positive definite preconditioner, MINRES minimizes a preconditioned residual norm over a Krylov subspace. For example, a suitable yet simple choice of preconditioner for MINRES here would be the block diagonal matrix $
\left[\begin{smallmatrix}
  \widetilde{\bo A} &\\
  & \widetilde{\bo H}
\end{smallmatrix}
\right]$.
Then, provided that $\widetilde{\bo H}$ is spectrally equivalent to $\bo S$, we may use the spectral bounds in~\cite{PestanaWathen2015} and convergence theory in \cite{Greenbaum1997} to show the robust convergence of MINRES. Notice that unlike the inexact Uzawa method above, MINRES does not require a suitably chosen damping parameter $\omega$ to have guaranteed convergence. Moreover, the condition~\eqref{eq:Aapprox_convergent} can be relaxed to the weaker assumption of spectral equivalence.
\end{remark}

\subsection{Proof of Theorem~\ref{thm:uzawa}.}\label{sec:uzawa_proof}

The proof essentially follows the approach in~\cite{Zulehner2002}, which gives sufficient conditions for convergence of inexact Uzawa methods in the context of general saddle-point problems with zero lower diagonal block.   However, the saddle-point matrix $\bm{\calA}$ defined in~\eqref{eq:mixed_formulation} has nonzero lower diagonal block, so, strictly speaking, we must check that the approach in \cite{Zulehner2002} can be extended to cover the present situation. Therefore, in this subsection, we adapt the main steps from \cite{Zulehner2002} for the sake of completeness. To remain brief, we do not attempt to give as general a treatment as the one in \cite{Zulehner2002}.

Let $\bm{\mathcal{M}}$ denote the iteration matrix of the inexact Uzawa method, i.e. $\bm{e}_{j+1} = \bm{\mathcal{M}} \bm{e}_j$ for each $j\geq 0$, where $\bm{e}_j\coloneqq \bm{u}-\bm{u}_j$. Furthermore, let $\bQ \coloneqq \frac{1}{\omega}\widetilde{\bo H}$ denote the rescaled Schur complement preconditioner. A simple calculation shows that
\begin{equation}
\bm{\mathcal{M}} = \begin{bmatrix}
  \bo I - \widetilde{\bo A}^{-1} \bo A & \widetilde{\bo A}^{-1} \bo K \\
  - \bQ^{-1} \bo K^\top \left(\bo I - \widetilde{\bo A}^{-1} \bo A \right) & \bo I - \bQ^{-1} \widetilde{\bo S}
\end{bmatrix},
\end{equation}
where $\widetilde{\bo S} \coloneqq \bo K^\top \widetilde{\bo A}^{-1} \bo K + \bo K+ \bo K^\top + \bo A$, and where $\bo I$ denotes the identity matrix on $\Vh^N$. Therefore, we can obtain~\eqref{eq:uzawa_contraction} by showing that $\norm{\bm{\mathcal{M}}}_{\bD} \leq \rhouz$, where $\norm{\bm{\mathcal{M}}}_{\bD}$ denotes the operator norm of $\mathcal{M}$ with respect to the norm $\bD$ from~\eqref{eq:Dnorm_def}.

\begin{lemma}\label{lem:uzawa_eigenvalue_problem}
Assume that \eqref{eq:Aapprox_convergent} holds, that $\omega>0$, and that $\widetilde{\bo H}$ is symmetric positive definite.
%Let the sequence of iterates $\{\bm{u}_j\}_{j\geq 0}$ be defined by~\eqref{eq:uzawa}, let the norm $\norm{\cdot}_{\bD}$ be defined as in~\eqref{eq:Dnorm_def},  
Then we have the bound $\norm{\bm{\calM}}_{\bD} \leq \max_{\mu \in \Sigma} \abs{\mu}$ where $\Sigma\subset \R$ denotes the set of eigenvalues of the generalized symmetric eigenvalue problem:
\begin{equation}\label{eq:uzawa_eigenvalue_problem}
\begin{aligned}
\bm{\calN} \bm{v} = \mu \, \bm{\calE} \bm{v},
&&& \bm{\calN} \coloneqq   \begin{bmatrix}
    \widetilde{\bo A} & - \bo K \\ - \bo K^\top & \widetilde{\bo S} - \bQ
  \end{bmatrix},
&&& \bm{\calE} \coloneqq \begin{bmatrix}
    \tfrac{1}{\rhoa} \widetilde{\bo A} & \\ & \bQ
  \end{bmatrix}
\end{aligned}
\end{equation}
where $\widetilde{\bo S} \coloneqq \bo K^\top \widetilde{\bo A}^{-1} \bo K + \bo K+ \bo K^\top + \bo A$.
\end{lemma}

\begin{proof}
Let $\bD \coloneqq  \left[\begin{smallmatrix} \omega  \rhoa \widetilde{\bo A} & \\ & \widetilde{\bo H} \end{smallmatrix}\right]$ denote the matrix inducing the norm~$\norm{\cdot}_{\bD}$ in \eqref{eq:Dnorm_def}.
A simple calculation shows that the matrix $\bm{\calM}$ can be factorized as $\bm{\calM} = - \bm{\calR} \bm{\calN} \bm{\calQ}$, where $\bm{\calN}$ is as in \eqref{eq:uzawa_eigenvalue_problem}, and $\bm{\calR} \coloneqq \left[ \begin{smallmatrix} \widetilde{\bo A}^{-1} & \\ & \bQ^{-1} \end{smallmatrix} \right]$,  $\bm{\calQ} \coloneqq \left[ \begin{smallmatrix} \widetilde{\bo A}^{-1}\bo A -\bo I & \\ & \bo I \end{smallmatrix} \right]$.
It is furthermore easy to check that $\omega \bm{\calE}^{-1} =  \bm{\calR}^\top \bD \bm{\calR}$, and that $\bm{\calQ}^\top \bm{\calE} \bm{\calQ} \leq \frac{1}{\omega} \bD$ as a consequence of~\eqref{eq:Aapprox_convergent}. Therefore, we find that
\begin{equation*}
\begin{split}
\norm{\bm{\calM}}_{\bD}^2 &= \sup_{ \bm{v}\in \Vh^N\times\Vh^N\setminus\{0\} } 
\frac{ \bm{v}^\top \bm{\calQ}^\top \bm{\calN} \bm{\calE}^{-1} \bm{\calN} \bm{\calQ} \bm{v} }{ \bm{v}^\top (\frac{1}{\omega}\bD) \bm{v} }\\
&\leq \sup_{\bm{w}\in \Vh^N\times\Vh^N\setminus\{0\}} \frac{ \bm{w}^\top \bm{\calN} \bm{\calE}^{-1} \bm{\calN} \bm{w} }{ \bm{w}^\top \bm{\calE} \bm{w}} 
= \norm{\bm{\calE}^{-1} \bm{\calN}}_{\bm{\calE}}^2,
\end{split}
\end{equation*}
where we obtain the inequality in second line above by substituting $\bm{w}=\bm{\calQ} \bm{v}$ and using the bound $\bm{\calQ}^\top \bm{\calE}\, \bm{\calQ} \leq \frac{1}{\omega} \bD$ given above.
Since the matrix $\bm{\calE}^{-1} \bm{\calN}$ is symmetric with respect to the $\bm{\calE}$-inner product, we see that $\norm{\bm{\calE}^{-1} \bm{\calN}}_{\bm{\calE}} = \max_{\mu \in \Sigma} \abs{\mu} $ where $\Sigma$ denotes the set of eigenvalues in \eqref{eq:uzawa_eigenvalue_problem}. This implies that $\norm{\bm{\calM}}_{\bD} \leq \max_{\mu \in \Sigma} \abs{\mu}$.
\end{proof}

The next step in the proof of Theorem~\ref{thm:uzawa} is to bound the eigenvalues $\mu\in \Sigma$ defined by \eqref{eq:uzawa_eigenvalue_problem}.

\begin{lemma}\label{lem:uzawa_eigenvalue_bounds}
Assume that \eqref{eq:Aapprox_convergent} holds, that $\omega>0$, and that $\widetilde{\bo H}$ is symmetric positive definite. Let $\lmax>0$ (respectively $\lmin>0$) be an upper bound on the maximum eigenvalue (respectively lower bound on the minimum eigenvalue) of $\widetilde{\bo H}^{-1} \bo S$.
If $\mu \in \Sigma$ is a negative eigenvalue of \eqref{eq:uzawa_eigenvalue_problem}, then $\abs{\mu}\leq \sigma_{-}$ where $\sigma_{-}$ is defined~\eqref{eq:sigma_min}.
If $\mu\in \Sigma$ is a positive eigenvalue, then $\abs{\mu} \leq \sigma_{+}$ with $\sigma_{+}$ defined in~\eqref{eq:sigma_plus}. 
\end{lemma}

\begin{proof}
We start by showing the bound $\abs{\mu}\leq \sigma_{-}$ for any negative eigenvalue $\mu<0$. First, let $\lambda = - \mu = \abs{\mu}$; then, since $(1+ \frac{\lambda}{\rhoa})\widetilde{\bo A}$ is positive definite, we may use the eigenvalue problem~\eqref{eq:uzawa_eigenvalue_problem} to find that there exists a nonzero $\bo v\in \Vh$ such that
\[
\lambda \bQ \bo v = \bQ \bo v - \widetilde{\bo S}\bo v + \frac{\rhoa}{\rhoa+\lambda } \bo K^\top \widetilde{\bo A}^{-1} \bo K \bo v.
\]
By taking the inner product with $\bo v$ and applying the inequalities~\eqref{eq:Aapprox_convergent} and $\widetilde{\bo S} \geq (1-\rhoa) \omega \lmin \bQ$, we eventually find that $\lambda$ satisfies the inequality
\begin{equation}
\lambda^2 - \lambda (1-\rhoa)(1-\omega \lmin) \leq \rhoa,
\end{equation}
from which we deduce that $\abs{\mu}=\lambda\leq \sigma_{-}$ as claimed.
  
Next we consider the case of positive eigenvalues $\mu>0$. First note that $\rhoa<\sigma_{+}$; therefore we need only consider the case of $\mu>\rhoa$, otherwise if $\mu\leq \rhoa$ then $\mu<\sigma_{+}$ and there is nothing left to show. If $\mu>\rhoa$, then $(1-\mu/\rhoa)\bQ$ is nonsingular, and we find that there is a nonzero $\bo v\in \Vh$ such that
\[
(1+\mu) \bQ \bo v = \widetilde{\bo S} \bo v + \frac{\rhoa}{\mu-\rhoa} \bo K^\top \widetilde{\bo A}^{-1}\bo K \bo v.
\]  
We then apply the inequality $\widetilde{\bo S}\leq (1+\rhoa) \omega \lmax \bQ$ to find eventually that
\begin{equation}
\mu^2 - \mu \left[ (1+\rhoa)(1+\omega \lmax)-2 \right] \leq \rhoa,  
\end{equation}
which implies the inequality $\mu \leq \sigma_{+}$ as claimed.
\end{proof}

\noindent\emph{Proof of Theorem~\ref{thm:uzawa}.}
The combination of Lemmas~\ref{lem:uzawa_eigenvalue_problem} and \ref{lem:uzawa_eigenvalue_bounds} implies that $\norm{\bm{\calM}}
\leq \rhouz$ where $\rhouz \coloneqq \max\{ \sigma_{-}, \sigma_{+} \}$; this implies the bound~\eqref{eq:uzawa_contraction} on the sequence of iterates.
We next show that $\rhouz <1$ if \eqref{eq:damping_condition} holds. Indeed, it is easy to check that $\sigma_{-}<1$ for all $\omega \lmin>0$, and $\sigma_{+} <1 $ if and only if $\rhoa + (1+\rhoa)(1+\omega \lmax) -2 < 1$, which is equivalent to the condition in~\eqref{eq:damping_condition}.
\qed

\section{Schur complement preconditioner}\label{sec:schur_prec}
We have seen in the previous section that a convergent iterative solver for~\eqref{eq:mixed_formulation} can be obtained provided that we have at our disposal a spectrally equivalent preconditioner for the Schur complement~$\bo S$. 
In this section, we propose such a preconditioner that is well-suited for parallel computations.
We shall denote this preconditioner by $\bo H$ in the case of exact spatial solvers, and by $\widetilde{\bo H}$ in the practical case of approximate spatial solvers.
To motivate our construction, we consider the following example.

\begin{example}[Uniform time-steps with constant coefficients]
Consider momentarily the case where $\tau_n = \tau$ and $A_n = A$ for all $n=1,\dots,N$, i.e.\ where the time-steps and spatial operators are constant in time. Then the Schur complement matrix $\bo S$ is
\[
\bo S = \frac{1}{\tau} K^\top K \otimes M A^{-1} M + (K+K^\top)\otimes M + \tau \Id_N \otimes A,
\]
where $K$ is from~\eqref{eq:algo_K_def}, where $\Id_N$ is the $N\times N$ identity matrix  and where $\otimes$ denotes the Kronecker product. It is then easy to see that the matrices $K^\top K$ and $K+K^\top$ are both symmetric positive definite, since
\begin{equation}
\begin{aligned}
K^\top K = \left( \begin{smallmatrix}
 2 & -1 &  & \\
 -1  & 2 & -1 &  \\
 &   & \ddots & -1\\
 &  & -1 & 1
 \end{smallmatrix} \right),
 &&& 
 K+K^\top = \left( \begin{smallmatrix}
 2 & -1 &  & \\
 -1  & 2 & -1 &  \\
 &   & \ddots & -1\\
 &  & -1 & 2
 \end{smallmatrix} \right).
 \end{aligned}
\end{equation}
As a special case of the results of Section~\ref{sec:infsup}, we will see that $\bo S$ is spectrally equivalent to a simpler matrix where the middle term $(K+K^\top)\otimes M$ is dropped.
Since this simpler matrix involves only a sum of two Kronecker products of matrices, it can be block-diagonalized with respect to time.
The key observation is that the matrix $K^\top K$ has explicitly known (generalized) eigenvalues and eigenvectors, which are related to discrete Sine transforms (DST), which are well-suited for time-parallelism.
 This suggests using the DST in time to obtain a block-diagonal and thus time-parallel preconditioner.
\qed
\end{example}

To define the preconditioners, we use the type-III DST, represented by the matrix $\bo\Phi$ that maps $\bo u\in \Vh^N $ to $\mathbf{\hat{u}}=\bo\Phi\,\bo u\in \Vh^N$, where  $\mathbf{\hat{u}} = [\hat{u}_1,\dots, \hat{u}_N]$ is defined by
\begin{equation}\label{eq:algo_DST}
\begin{aligned}
\hat{u}_k \coloneqq \frac{2}{N} \sum_{n=1}^N \frac{1}{1+\delta_{nN}} u_n \sin\left( \frac{(2k-1) n\pi}{2N} \right), &&& k=1,\dots, N.
\end{aligned}
\end{equation}
Note that each coefficient $\hat{u}_k\in \Vh$, and that the weight-term $(1+\delta_{nN})^{-1}$ in \eqref{eq:algo_DST} is equal to $1$ for all $n<N$ and is equal to $1/2$ if $n=N$. The inverse map $\bo\Phi^{-1}$ that satisfies $\bo u = \bo\Phi^{-1} \mathbf{\hat{u}}$ is simply given by the type-II DST %particular, the action of $\bo\Phi^{-1}$ is given by
\begin{equation}\label{eq:algo_DST_inverse}
 u_n = \sum_{k=1}^N \hat{u}_k   \sin\left( \frac{(2k-1) n\pi}{2N} \right), \quad n=1,\dots, N.
\end{equation}
We stress that the actions of the transformations are with respect to the temporal components of the vectors $\bo u$ and $\bo{\hat{u}}$, since each term in the sums of \eqref{eq:algo_DST} and \eqref{eq:algo_DST_inverse} is a vector in $\Vh$. Thus the DST used here represents a change of the temporal basis.
Furthermore, the actions of the matrices $\bo\Phi$ and $\bo\Phi^{-1}$ (and their transposes) can all be implemented efficiently through recursive splittings of the summations, leading to fast implementations akin to the FFT.

The ideal preconditioner $\bo H$ is defined by
\begin{equation}\label{eq:algo_H_def}
\begin{aligned}
\bo H \coloneqq \bo\Phi^{\top} \mathbf{\bo{\hat{H}}}\, \bo\Phi, &&&
\mathbf{\bo{\hat{H}}} \coloneqq
\frac{N}{2\tau } \Diag \left\{ H_k  A^{-1} H_k \right\}_{k=1}^N,
\end{aligned}
\end{equation}
where $\tau$ and $A$ are as in \eqref{eq:tauA_assumption}, and where the matrices $H_k$ are defined by
\begin{equation}\label{eq:Hk_def}
\begin{aligned}
H_k \coloneqq \mu_k M + \tau A,  && \mu_k \coloneqq 2 \sin\left( \frac{(2k-1)\pi}{4N} \right) && \forall\, k=1,\dots,N.
\end{aligned}
\end{equation}
Notice that $\mu_k>0$ for each $1\leq k \leq N$.
The inversion of $\bo H$ can be performed by composition of $\bo{\Phi}^{-T}$, $\bo{\hat{H}}^{-1}$ and $\bo{\Phi}^{-1}$. 
As mentioned above, the actions of $\bo{\Phi}^{-T}$ and $\bo{\Phi}^{-1}$ can be computed by fast DST algorithms akin to the FFT, and the application of $\bo{\hat{H}}^{-1}$ simply requires the solution of linear systems for weighted implicit Euler steps and can be parallelized over the blocks $k=1,\dots,N$.

\subsection{Approximations}
The analysis of iterative solvers given below will allow for approximations to be made in the application of the inverse of $\bo H$. 
More precisely, we consider approximations~$\widetilde{\bo H}$ of $\bo H$ given by
\begin{equation}\label{eq:Happrox_def}
\begin{aligned}
\widetilde{\bo H}\coloneqq \bo{\Phi}^\top \hat{\bo{H}}_{\mathrm{approx}} \bo{\Phi},
&&&
\hat{\bo{H}}_{\mathrm{approx}} \coloneqq \frac{N}{2\tau }  \Diag \left\{ \widetilde{H}_k A^{-1} \widetilde{H}_k \right\}_{k=1}^N,
\end{aligned} 
\end{equation}
where, for each $k=1,\dots,N$, the symmetric positive definite matrix $\widetilde{H}_k$ represents an approximation of $H_k$.
We have in mind cases where each matrix $\widetilde{H}_k$ is obtained from a standard solver for the matrices $H_k$, for example by multigrid or domain decomposition methods. 
For the analysis, we shall assume that there exist positive constants $\gmin$ and $\gmax$ such that, for all $k=1,\dots,N$,
\begin{equation}\label{eq:Hk_spectral_equivalence}
\begin{aligned}
\gmin \, \widetilde{H}_k A^{-1} \widetilde{H}_k \leq H_k A^{-1} H_k \leq \gmax \,\widetilde{H}_k A^{-1} \widetilde{H}_k.
\end{aligned}
\end{equation}
By comparing~\eqref{eq:Happrox_def} with \eqref{eq:algo_H_def}, it is then clear that the matrices $\bo H$ and $\widetilde{\bo H}$ are spectrally equivalent, with $\gmin \widetilde{\bo H} \leq \bo H \leq \gmax \widetilde{\bo H}$.

\subsection{Spectral bounds for the Schur complement preconditioners}
Our main result is that the preconditioner $\bo H$ defined in~\eqref{eq:algo_H_def}, and its approximation $\widetilde{\bo H}$ defined in~\eqref{eq:Happrox_def}, are spectrally equivalent to $\bo S$ with known constants in the bounds.

\begin{theorem}[Spectral equivalence]\label{thm:general_spectral_equivalence}
Assume~\eqref{eq:tauA_assumption}. Then, for any number of time-steps $N$ and any symmetric positive definite matrices $M$ and $\{A_n\}_{n=1}^N$, the matrices $\bo H$ and $\bo S$ are spectrally equivalent with the following bounds
\begin{equation}\label{eq:general_SH_equiv}
\begin{aligned}
\frac{1}{2\alpha} \,\bo H \leq \bo S \leq 3 \alpha \, \bo H.
%\frac{1}{2  \alpha \, } \leq \frac{\bo v^\top \bo S \,\bo v}{\bo v^\top \bo H \, \bo v} \leq 3  \alpha    &&& \forall\, \bo v \in \Vh^N \setminus\{0\},
\end{aligned}
\end{equation}
Furthermore, if \eqref{eq:Hk_spectral_equivalence} also holds, then $\widetilde{\bo H}$ defined in \eqref{eq:Happrox_def} is spectrally equivalent to~$\bo S$ with the following bounds
\begin{equation}\label{eq:SHapprox_equiv}
\begin{aligned}
\frac{\gmin}{2 \alpha} \, \widetilde{\bo H} \leq \bo S \leq 3 \alpha  \gmax  \, \widetilde{\bo H}.
\end{aligned}
\end{equation}
\end{theorem}

The proof of Theorem~\ref{thm:general_spectral_equivalence} is the subject of Sections~\ref{sec:infsup} and~\ref{sec:schur}.
Note that the constants appearing in the bounds of Theorem~\ref{thm:general_spectral_equivalence} can thus be substituted for the constants $\lmin$ and $\lmax$ in Theorem~\ref{thm:uzawa}.

\begin{remark}[Robust convergence]
The combination of Theorems~\ref{thm:uzawa} and~\ref{thm:general_spectral_equivalence} leads to bounds on the convergence rate of the method that depend only on the the constants in the assumptions~\eqref{eq:tauA_assumption}, \eqref{eq:Aapprox_convergent}, \eqref{eq:Hk_spectral_equivalence}, and on $\omega$. 
In many practical applications, these assumptions are satisfied with uniformly bounded constants independent of parameters that determine the spatial matrices $M$ and $A_n$, such as the spatial mesh size, in which case the convergence will be robust.
\end{remark}

\section{Parallel complexity}\label{sec:parallel_complexity}
Following~\cite{HortonVandewalle1995,HortonVandewalleWorley1995}, the notion of parallel complexity is understood here as the theoretical computational cost assuming the availability of sufficiently many processors, and ignoring communication costs. It is therefore of interest as an intrinsic property of the given algorithm.
Since we are primarily interested in the time-parallelism of the algorithm, we shall focus on the dependence on the number of time-steps $N$.

In order to treat the costs related to spatial operations related to $\Vh$ in a general way, we introduce the following elementary constants.
\begin{itemize}
  \item  Let $\Cadd$ denote the cost of additions and subtractions of vectors in $\Vh$; more precisely, $\Cadd$ is the maximal cost of the operation $(v,w, c)\mapsto  v + c \, w$, where $ v$, $ w \in \Vh$ and $c\in \R$.
  \item Let $\Cmult$ denote the maximal cost of performing a matrix vector product $(L, v)\mapsto L\,v$, where $v\in \Vh$ and the matrix $L$ is one of $M$, $A$ or $A_n$, $n=1,\dots,N$.
  \item Let $\Csol$ denote the maximal cost of performing the action of the spatial preconditioners, i.e.\ the cost of the matrix-vector product $(L,v)\mapsto L\,v$ where $L$ is one of the $\widetilde{A_n}{}^{-1}$ or $\widetilde{H_k}{}^{-1}$, for $n, k =1,\dots, N$.
\end{itemize}
Given that the proposed algorithm allows the re-use of existing spatial solvers, it is clear that in many applications there can be significant spatial parallelism as well; see the experiments in Section~\ref{sec:numexp4}.
We distinguish the costs of these different operations since in practice they may be rather different; for instance, we expect that $\Cadd$ will be smaller than $\Cmult$ or $\Csol$.

We can now analyse the parallel complexity of the inexact Uzawa method using $\widetilde{\bo H}$  as defined in \eqref{eq:Happrox_def}.
For fixed constants $\rhoa$, $\omega$, $\alpha$, $\gmin$ and $\gmax$, the convergence rate of the algorithm is robust with respect to the number of time-steps. 
Therefore, for any $\eps>0$, at most $O(\log\eps^{-1})$ iterations are required to achieve a relative reduction of the residual by a tolerance $\eps$. The total cost is then based on the number of iterations required multiplied by the cost per iteration.
Each iteration of the inexact Uzawa method~\eqref{eq:uzawa} requires
\begin{itemize}
\item A fixed number of matrix vector products with $\bo K$, $\bo A$ and $\bo K^\top$, each of which has parallel complexity $O(\Cmult + \Cadd)$, independently of $N$, since each of these matrices is block-sparse.
  \item a fixed number of vector additions/subtractions on $\Vh^N\times\Vh^N$, with parallel complexity $O(\Cadd)$, independently of $N$.
  \item The application of $\widetilde{\bo A}^{-1}$, which has parallel complexity $O(\Csol)$, independently of $N$, since $\widetilde{\bo A}$ is block diagonal.
  \item The application of $\widetilde{\bo H}^{-1}$, which we discuss further below.
\end{itemize}

The matrix-vector product with $\widetilde{\bo H}{}^{-1}$ involves the application of the DST transformations related to $\bo \Phi^{-1}$ and $\bo \Phi^{-\top}$, and the application of the block diagonal matrix $\bo{\hat{H}}_{\mathrm{approx}}^{-1}$ from \eqref{eq:Happrox_def}.
It is clear that the application of~$\bo{\hat{H}}_{\mathrm{approx}}^{-1}$ has parallel complexity $O(\Cmult+\Csol)$, and is independent of $N$. 
It remains only to consider the parallel complexity of the DST. It is clear that if $N$ is, for example, an integer power of $2$, then the DST has parallel complexity of $O(\Cadd (\log N+1))$, as shown by recursive splitting of the summation in \eqref{eq:algo_DST_inverse}.
From a theoretical perspective, the same parallel complexity bound can also be achieved for general $N$, since a DST of general length~$N$ can be obtained by Bluestein's method, which involves two zero-padded discrete Fourier transforms with length equal to a power of two of same order as $N$; see \cite{Pelz1997} for further details.

In summary, the parallel complexity of each inexact Uzawa iteration is then bounded by
\begin{equation}\label{eq:complexity_estimate}
O\left ( \Cadd ( \log N + 1 ) + \Cmult + \Csol \right).  
\end{equation}
It is thus seen that each iteration has a parallel complexity that  grows at most logarithmically with $N$. 
The strong decoupling of the method between time and space can also be seen through the fact that the terms involving $\Cmult$ and $\Csol$ are independent of $N$.
Note that the terms of order $\log N$ are not necessarily dominant in actual computations, since $\Csol$ and $\Cmult$ are often significantly larger than $\Cadd$; indeed, in our experiments, the cost of the DST is significantly lower than the cost of solving the associated spatial problems, see Section~\ref{sec:numexp3} below for further details.

\begin{remark}[Comparison with the parareal method]
As mentioned above, it is known that the best parallel complexity of the parareal algorithm using the implicit Euler scheme grows as $\sqrt{N}$, see~\cite{BalMaday2002}. More precisely, in the current notation, it can shown to be of order $O((\Csol+\Cadd+\Cmult) \sqrt{N} )$. It is then seen that the difference with \eqref{eq:complexity_estimate} is not only in the order of dependence on $N$, but also in the associated constants due to the spatial problems.
\end{remark}

\section{Inf-sup stability of the implicit Euler method}\label{sec:infsup}
As mentioned in the introduction, the derivation and analysis of the proposed algorithm is strongly tied to the inf-sup stability of the problem. Therefore, we aim to offer in this section a clear conceptual understanding of the algorithm by detailing  the inf-sup analysis and explaining the physical significance of the left-preconditioner used to obtain the equivalent formulations~\eqref{eq:algo_S_system} and \eqref{eq:mixed_formulation}. Furthermore, the results given in this section prepare the ground for the proof of Theorem~\ref{thm:general_spectral_equivalence}.
In particular, we shall work with the interpretation of the implicit Euler method as the lowest-order discontinuous Galerkin time-stepping method, where the system matrix $\bo B$ admits a representation as a bilinear form on spaces of functions that are piecewise constant-in-time.
The advantage of this approach using bilinear forms is that it establishes the connection between the left-preconditioned matrix $\bo S$ and the underlying physical parabolic norm that it represents.

\subsection{Time-global variational formulation}\label{sec:euler_time_global}
Define the space
\begin{equation}\label{eq:Vt_def}
\Vt \coloneqq \oplus_{n=1}^N \calP_{0}(I_n;\Vh),
\end{equation}
where $\calP_{0}(I_n;\Vh)$ denotes the set of $\Vh$-valued functions that are constant-in-time over each time-step interval $I_n$, for $n=1,\dots,N$.
In other words, a function $v \in \Vt$ if and only if $v$ is a piecewise constant function on each time-step $I_n$, with $v|_{I_n}\in \Vh$.
A basis can be constructed for $\Vt$ by considering the tensor product between a basis of $\Vh$ and a basis for $\oplus_{n=1}^N \calP_{0}(I_n,\R)$ the space of real-valued piecewise-constant functions. Thus, a standard choice for the temporal basis of $\Vt$ is given as follows: for any $v\in\Vt$, we have $v = \sum_{n=1}^N v_n \chi_n$, where the coefficients $v_n\in \Vh$, and where $\chi_n$ is the indicator function of $I_n$, for each $n=1,\dots,N$; thus, $v_n=v|_{I_n}$ the restriction of $v$ to $I_n$.

Define the reconstruction operator $\calI \colon \Vt \tends H^1(0,T;\Vh) \cap \oplus_{n=1}^N \calP_{1}(I_n;\Vh)$ by
\begin{equation}\label{eq:time_reconstruction}
(\calI v)(t) \coloneqq v_n - \frac{t_n-t}{\tau_n}\lld v \rrd_{n-1}, \quad t\in (t_{n-1},t_n],
\end{equation}
where $\calP_{1}(I_n;\Vh)$ denotes the set of piecewise affine functions on each time-step $I_n$, and where $\lld \cdot \rrd$ denotes the jump operator defined by $\lld v \rrd_{n-1} \coloneqq v_n-v_{n-1}$, with the convention that $v_0=0$ for all $v\in \Vt$ to simplify the notation.
It follows that $\calI$ defines a linear operator on $\Vt$, and that, for any $v\in \Vt$, the function~$\calI v$ is a piecewise-affine continuous function in time, with $\calI v(t_{n}) = v_n$ for each $n=1,\dots,N$.
Thus the function $\calI v \in H^1(0,T;\Vh)$, and $\calI v$ has a weak temporal derivative $\p_t \calI v\in \Vt$ with $(\p_t \calI v)|_{I_n} = \frac{1}{\tau_n}(v_n-v_{n-1})$ for all $n=1,\dots,N$.
Additionally, the function $\calI v$ also satisfies the initial condition $\calI v (0) =0$.

The implicit Euler method~\eqref{eq:algo_euler} can then be equivalently rewritten as: find $u\in\Vt$ such that
\begin{equation}\label{eq:scheme}
\begin{aligned}
b(u,v) = \ell(v) & & &\forall\,v\in \Vt,
\end{aligned}
\end{equation}
where the bilinear form $b(\cdot,\cdot)$ and linear functional $\ell(\cdot)$ are defined by
\begin{equation}\label{eq:b_form_def}
\begin{aligned}  
b(u,v)  \coloneqq \sum_{n=1}^N \int_{I_n} (\p_t \calI u,v)_{M} + (u,v)_{A_n} \,\dd t, &&&
\ell(v) \coloneqq  (u_{I},v_1)_{M} + \sum_{n=1}^N \int_{I_n}(f_n,v)_{M} \,\dd t,
\end{aligned}
\end{equation}
where we recall that $u_{I}$ is the given initial datum.
The matrix $\bo B$ that represents the bilinear form $b(\cdot,\cdot)$ in the standard basis of $\Vt$ is given in \eqref{eq:algo_B_system}.

\subsection{Inf-sup stability}\label{subsec:infsup}
As explained above, for many parabolic equations, the natural norm for the temporal derivative is a dual norm induced by the spatial differential operator of the problem. In the discrete setting, these dual norms admit the following characterizations.
First, we introduce the dual norms $\norm{\cdot}_{M A_n^{-1} M}$, for each $n=1,\dots, N$, and $\norm{\cdot}_{M A^{-1} M}$ on $\Vh$, defined by
\begin{equation}
\begin{aligned}
\norm{w}_{M A_n^{-1} M} \coloneqq \sup_{v \in\Vh\setminus\{0\}}\frac{(w,v)_{M}}{\norm{v}_{A_n}}, &&&  \norm{w}_{M A^{-1} M} \coloneqq \sup_{v \in\Vh\setminus\{0\}}\frac{(w,v)_{M}}{\norm{v}_{A}}  &&&\forall\,w\in \Vh.
\end{aligned}
\end{equation}
Since $M$ is positive definite, and since $\Vh$ is finite dimensional, it is clear that $\norm{\cdot}_{M A^{-1} M}$ and $\norm{\cdot}_{M A^{-1}_n M}$ define norms on $\Vh$. 
The notation $\norm{\cdot}_{M A_n^{-1}M}$, is justified by the fact that this norm is induced by the inner product of the matrix $M A_n^{-1} M$, since it is straightforward to show that $\norm{w}_{MA_n^{-1}M}^2 = w^\top M A_n^{-1} M w$ and $\norm{w}_{MA^{-1}M}^2 = w^\top M A^{-1} M w$ for all $w\in \Vh$, where we again identify $w \in \Vh$ with its vector representation.
It is easy to show that the spectral equivalence~\eqref{eq:tauA_assumption} implies that $\frac{1}{\tau_n} M A^{-1}_n M$ and $\frac{1}{\tau} M A^{-1} M$ are also spectrally equivalent:
\begin{equation}\label{eq:An_dual_spectral_equivalence}
  \begin{aligned}
  \frac{1}{\alpha} M (\tau A)^{-1} M \leq M (\tau_n A_n)^{-1} M \leq \alpha\, M (\tau A)^{-1} M.
  \end{aligned}
\end{equation}
We introduce the following norms on the space $\Vt$
\begin{subequations}\label{eq:discrete_norms}
\begin{align}
\norm{v}_{\X}^2 &\coloneqq \sum_{n=1}^N \int_{I_n} \norm{v}_{A_n}^2 \,\dd t , \\
\norm{u}_{\Y}^2 &\coloneqq \sum_{n=1}^N \int_{I_n} \norm{\p_t \calI u}_{M A_n^{-1} M}^2 + \norm{u}_{A_n}^2 \,\dd t + \norm{u_N}_{M}^2 +  \sum_{n=1}^N \norm{\lld u \rrd_{n-1}}_{M}^2 ,
\end{align}
\end{subequations}
for all functions $u$ and $v$ in $\Vt$, where the jump operators $\lld \cdot \rrd$ were defined in Section~\ref{sec:euler_time_global}.

The following theorem shows that the inf-sup stability of the discrete problem holds with constant equal to one for the norms defined in \eqref{eq:discrete_norms}; this constitutes a sharp characterization of the stability of the discrete problem.

\begin{theorem}[Inf-sup stability of the Implicit Euler method]\label{thm:discrete_inf_sup}
The bilinear form $b(\cdot,\cdot)$ be defined in~\eqref{eq:b_form_def} is inf-sup stable with respect to the norms in~\eqref{eq:discrete_norms}. For any function $u\in \Vt$, we have
\begin{equation}\label{eq:discrete_inf_sup}
\begin{aligned}
\norm{u}_{\Y} = \sup_{v\in\Vt\setminus\{0\}} \frac{b(u,v)}{\norm{v}_{\X}} .
\end{aligned}
\end{equation}
Moreover, for each $u\in \Vt$, the supremum in \eqref{eq:discrete_inf_sup} is achieved by the optimal test function $v_*\in \Vt$ given by
\begin{equation}\label{eq:optimal_test_function}
\begin{aligned}
v_*|_{I_n} \coloneqq   A_n^{-1} M (\p_t \calI u)|_{I_n} + u|_{I_n}, &&& \forall\,n=1,\dots,N.
\end{aligned}
\end{equation}
\end{theorem}

\noindent The proof of theorem~\ref{thm:discrete_inf_sup} will be given below.

For the purposes of preconditioning, we can exploit the explicit knowledge of the optimal test function given in Theorem~\ref{thm:discrete_inf_sup}. We define the operator $P \colon \Vt \tends \Vt$ that maps a function $v$ to its optimal test function, namely
\begin{equation}\label{eq:P_def}
\begin{aligned}
P v|_{I_n} \coloneqq A_n^{-1} M ( \p_t \calI v )|_{I_n} + v|_{I_n} & & & \forall\, v\in \Vt.
\end{aligned}
\end{equation}
As explained above, the time derivative $\p_t \calI v$ is indeed in $\Vt$ owing to the facts that $\calI v_{h}\in H^1(0,T;\Vh)$ and $\calI v$ is piecewise affine. Therefore, we indeed have $P v \in \Vt$ for all $v\in \Vt$. It is also possible to show that $P$ is invertible on $\Vt$; hence, the numerical scheme \eqref{eq:scheme} is equivalent to finding $u\in \Vt$ such that
\begin{equation}\label{eq:scheme_left_preconditioned}
\begin{aligned}
s(u,v) = g(v), & & s (u,v) \coloneqq b (u, P v ), && g(v) \coloneqq \ell(P v) & & \forall\, v\in\Vt.
\end{aligned}
\end{equation}
In the standard basis of $\Vt$, the matrix $\bo S$ associated to the bilinear form $s(\cdot,\cdot)$ can be written as in \eqref{eq:algo_S_system}, with the left-preconditioner matrix $\bo P^{\top}$ related to the operator~$P$.
Furthermore we have the identity $\bo g = \bo P^{\top} \bo f$, where $\bo g$ and $\bo f$ are the vectors representing the actions of the linear functionals $g$ and $\ell$ on the basis functions. Therefore, the symmetric formulation of the problem~\eqref{eq:scheme_left_preconditioned} in terms of bilinear forms is equivalent to its matrix formulation~\eqref{eq:algo_S_system}.

\begin{theorem}[Symmetrization by left-preconditioner]\label{thm:symm}
The bilinear form $s(\cdot,\cdot)$ defined in~\eqref{eq:scheme_left_preconditioned} is symmetric, positive definite, and satisfies the following identity: for all $u$ and $v$ in $\Vt$,
\begin{equation}\label{eq:S_identity}
\begin{split}
s(u,v) = \sum_{n=1}^N \int_{I_n} (\p_t \calI u, \p_t \calI v)_{M A_n^{-1} M} + (u,v)_{A_n}\,\dd t + j(u,v),
\end{split}
\end{equation}
where the symmetric bilinear form $j(\cdot,\cdot)$ is defined by
\begin{equation}\label{eq:J_def}
j(u,v) \coloneqq (u_N,v_N)_{M} + \sum_{n=1}^N ( \lld u\rrd_{n-1}, \lld v\rrd_{n-1} )_{M}.
\end{equation}
The unique solution $u$ of \eqref{eq:scheme} is equivalently the unique solution of \eqref{eq:scheme_left_preconditioned}, and we have the identity $s(v,v) = \norm{v}_{\Y}^2$ for any $v \in \Vt$.
\end{theorem}
\begin{proof}
We simply show \eqref{eq:S_identity} by calculation: for arbitrary $u$, $v\in \Vt$, $s(u,v) = b(u,P v) = \sum_{n=1}^N \int_{I_n} (\p_t \calI u, P v)_{M} + (u,P v)_{A_n}\,\dd t$.
By expanding the terms and simplifying, we obtain
\begin{equation}
s(u,v) = \sum_{n=1}^N\int_{I_n} (\p_t \calI u,\p_t \calI v)_{M A_n^{-1}M}+(u,v)_{A_n} + (\p_t \calI u,v)_{M} + (u,\p_t \calI v)_{M}\,\dd t,
\end{equation}
where we have made use of the identities $\big( \p_t\calI u,A_{n}^{-1} M \p_t\calI v\big)_M = ( \p_t\calI u ,  \p_t\calI v)_{M A_n^{-1} M}$ and $\big(u,A_n^{-1} M  \p_t\calI v\big)_{A_n}   = (u, \p_t\calI v)_M$ for any $u,\,v \in \Vt$.
To complete the proof of \eqref{eq:S_identity}, we now show that
\begin{equation}\label{eq:integration_identity}
 \sum_{n=1}^N \int_{I_n} (\p_t \calI u,v)_{M} + (u,\p_t \calI v)_{M}\,\dd t  =  j(u,v),
\end{equation}
where $j(\cdot,\cdot)$ is defined as in \eqref{eq:J_def}. Indeed, recall that $\p_t \calI u|_{I_n} = \frac{1}{\tau_n}(u_n - u_{n-1})$ for each $1\leq n \leq N$, so we obtain $\int_{I_n} (\p_t \calI u,v)_{M} + (u,\p_t \calI v)_{M}\,\dd t
= (\lld u \rrd_{n-1}, \lld v \rrd_{n-1})_{M} + (u_n,v_n)_{M}-(u_{n-1},v_{n-1})_{M}$, where we obtained the identity by adding and subtracting $(u_n-u_{n-1},v_{n-1})_{M}$.
Then, the identity in~\eqref{eq:integration_identity} is obtained by simplifying the telescoping sum.
\end{proof}

\noindent\emph{Proof of Theorem~\ref{thm:discrete_inf_sup}.}
We start by noting that $(\p_t \calI u,v)_{M} = (A_{n}^{-1} M \p_t \calI u, v)_{A_n}$ for all $v\in \Vh$. Hence, we may write $b(u,v) = \sum_{n=1}^N \int_{I_n} (P u, v)_{A_n} \,\dd t$, where the operator $P$ is defined in \eqref{eq:P_def}. This immediately implies that $\norm{P u }_{\X} = \sup_{v\in \Vt\setminus\{0\}} \frac{b(u,v)}{\norm{v}_{\X}}$ with equality achieved by the test function $v = Pu$. Hence, the identity \eqref{eq:S_identity} shows that $\norm{P u }_{\X}^2 = b(u,Pu) = s(u,u) = \norm{u}_{\Y}^2$, which implies \eqref{eq:discrete_inf_sup}.
\qed

A consequence of Theorem~\ref{thm:symm} is that the norm~$\norm{\cdot}_{\Y}$ for functions in $\Vt$ coincides with the norm $\norm{\cdot}_{\bo S}$ for $\Vh^N$, i.e.~we have $\norm{v}_{\Y} = \norm{\bo v}_{\bo S}$ for all $v\in \Vt$, where $\bo v\in \Vh^N$ is the vector representation of $v$. Furthermore, this shows the equivalence between \eqref{eq:infsup_matrix} and \eqref{eq:discrete_inf_sup}.

We now prove the max-norm bound mentioned above in Remark~\ref{rem:max_norm}.

\begin{corollary}[Max-norm bound]\label{cor:max_norm_bound}
For any $u\in \Vt$, with vector representation $\bo u \in \Vh^N$, we have
\begin{equation}\label{eq:max_norm}
  \max_{1\leq n \leq N}\norm{u_n}_M \leq \norm{u}_{\Y} = \norm{\bo u}_{\bo S}.
\end{equation}
\end{corollary}
\begin{proof}
Let $1\leq n\leq N$ be arbitrary. Define the test function $v\in \Vt$ by $v_m = u_m$ for all $m\leq n$, and $v_m =0$ for all $m>n$. Then, it is straightforward to show that $b(u,v)= \frac{1}{2}\norm{u_n}^2_M + \frac{1}{2}\sum_{m=1}^n \norm{\lld u\rrd_{m-1}}_M^2 + \norm{v}_{\X}^2$. We then use the inf-sup identity~\eqref{eq:discrete_inf_sup} and Young's inequality to find that $\abs{b(u,v)}\leq \frac{1}{2} \norm{u}_{\Y}^2 + \frac{1}{2}\norm{v}_{\X}^2$, which yields the bound $\norm{u_n}_M \leq \norm{u}_{\Y}$; this completes the proof since $n$ was arbitrary.
\end{proof}

\section{Spectral equivalence of the Schur complement preconditioner}\label{sec:schur}

In this section we analyse the spectral equivalence between the preconditioner $\bo H$, defined in \eqref{eq:algo_H_def}, and the Schur complement matrix $\bo S$, leading to the proof of Theorem~\ref{thm:general_spectral_equivalence}.
The first ingredient for the analysis, shown in~Lemma~\ref{lem:jump_bounds} below, is the fact that the jump terms in the bilinear form $j(\cdot,\cdot)$ can be controlled by the other terms in $s(\cdot,\cdot)$ with fully robust constants that do not depend on $N$ or on the spatial matrices.

\begin{lemma}[Bound on the jump terms]\label{lem:jump_bounds}
Let $\{\eps_n\}_{n=1}^N \subset \R_{>0}$ be an arbitrary collection of positive real numbers. Then, for any $v\in \Vt$, there holds
\begin{equation}
\begin{aligned}
0\leq j(v,v) \leq \sum_{n=1}^N  \int_{I_n} \frac{1}{\eps_n} \norm{\p_t \calI v}_{M A_n^{-1} M}^2 +  \eps_n \norm{v}_{A_n}^2\,\dd t. \label{eq:J_bound}    
\end{aligned}
\end{equation}
\end{lemma}
\begin{proof}
We find that $0\leq j(v,v)  \leq 2 \sum_{n=1}^N \int_{I_n} \norm{\p_t \calI v}_{MA_n^{-1}M}\norm{v}_{A_n}\,\dd t$ as a result of the definition of $\norm{\cdot}_{M A_n^{-1} M}$ in \eqref{eq:integration_identity}; then  \eqref{eq:J_bound} follows from Young's inequality.
\end{proof}

In the analysis below, we will use the auxiliary bilinear form $s_D\colon \Vt \times \Vt \tends \R$ defined by
\begin{equation}\label{eq:D_def}
s_D(u,v) \coloneqq \sum_{n=1}^N \left\{ \int_{I_n} (\p_t \calI u,\p_t\calI v)_{M A_n^{-1}M} \,\dd t + \frac{1}{1+\delta_{nN}} \int_{I_n} (u,v)_{A_n}\,\dd t \right\},
\end{equation}
where it is recalled that the weight-term $(1+\delta_{nN})^{-1}$ is equal to $1$ for all $1\leq n<N$ and is equal to $1/2$ if $n=N$. Thus it is seen that $s_D(\cdot,\cdot)$ is closely related to $s(\cdot,\cdot)$, where the jump bilinear form $j(\cdot,\cdot)$ has been removed, and a weight is included in one of the terms. This weight is needed later for the analysis of the DST below, see Remark~\ref{rem:weight_term}.
An almost immediate consequence of Lemma~\ref{lem:jump_bounds} is that $s_D(\cdot,\cdot)$ and $s(\cdot,\cdot)$ are spectrally equivalent with fully robust constants; thus $s_D(\cdot,\cdot)$ represents the dominant terms in $s(\cdot,\cdot)$.

\begin{lemma}[Spectral equivalence]\label{lem:spectral_SD}
Let $s_D$ be defined in \eqref{eq:D_def}. Then, we have
\begin{equation}\label{eq:SD_spectral_equiv}
\begin{aligned}
  s_D(v,v) \leq s(v,v) \leq 3 \, s_D(v,v) &&& \forall\, v\in \Vt.
\end{aligned}
\end{equation}
\end{lemma}
\begin{proof}
The lower bound $s_D(v,v) \leq s(v,v)$ for all $v\in \Vt$ follows directly from the fact that $j(v,v)\geq 0$ as shown by~\eqref{eq:J_def}. Next, we use Lemma~\ref{lem:jump_bounds} with $\eps_{n}\coloneqq(1+\delta_{nN})^{-1}$ to obtain
\begin{equation*}
\begin{split}
s(v,v) &\leq \sum_{n=1}^N \left\{ \left(1+\frac{1}{\eps_n}\right)\int_{I_n}\norm{\p_t \calI v}_{M A_n^{-1}M}^2\,\dd t + (1+\eps_n) \int_{I_n}\norm{v}_{A_n}^2\,\dd t \right\}
\\
&= \sum_{n=1}^N  \left(2+\delta_{nN}\right) \left\{\int_{I_n}\norm{\p_t \calI v}_{M A_n^{-1}M}^2\,\dd t + \frac{1}{1+\delta_{nN}} \int_{I_n}\norm{v}_{A_n}^2\,\dd t \right\},
\end{split}
\end{equation*}
which implies $s(v,v) \leq 3 \, s_D(v,v)$ for all $v\in\Vt$.
\end{proof}

\subsection{Discrete Sine Transform} Let the piecewise constant real-valued functions $\{\vphi_{k}\}_{k=1}^N \subset \oplus_{n=1}^N \calP_0(I_n;\R)$ be defined by
\begin{equation}\label{eq:vphi_def}
\begin{aligned}
  \vphi_k|_{I_n} = \vphi_k^n \coloneqq \sin\left( \frac{(2k-1) n\pi}{2N} \right), \qquad n=1,\dots, N.
\end{aligned}
\end{equation}
It is important to note that the functions $\vphi_k$ are globally supported in time. Moreover, we also define $\vphi_k^0 \coloneqq 0$ for all $1\leq k \leq N$. These functions are linearly independent and form a basis of $\oplus_{n=1}^N \calP_0(I_n;\R)$, so any function $v \in \Vt$ can be written in the form $v = \sum_{k=1}^N \hat{v}_k\, \vphi_k$ with coefficients $\hat{v}_k \in \Vh$ for each $k=1,\dots,N$.
The following result shows several basic properties of the basis induced by $\{\vphi_k\}_{k=1}^N$.

\begin{lemma}[Discrete orthogonality]\label{lem:orthogonality}
The functions $\{\vphi_k\}_{k=1}^N$ be defined by \eqref{eq:vphi_def}.
form an orthogonal basis of $\oplus_{n=1}^N \calP_0(I_n;\R)$ in the following discrete inner-products: for any $1\leq k,\,j \leq N$, we have
\begin{equation}\label{eq:L2_orthog}
\begin{aligned}
\sum_{n=1}^N\frac{1}{1+\delta_{nN}}  \vphi^n_k \, \vphi^n_j  = \frac{N}{2} \delta_{kj}, &&&
\sum_{n=1}^N (\vphi^n_k-\vphi^{n-1}_k)\,(\vphi^n_j - \vphi^{n-1}_j) = \frac{N}{2} \mu_k^2 \delta_{kj},
\end{aligned}
\end{equation}
where $\delta_{kj}$ is the Kronecker delta, and where the positive real numbers $\mu_k$ are defined in~\eqref{eq:Hk_def}. For any function $v\in \Vt$, we have the change of basis formulas given in \eqref{eq:algo_DST} and \eqref{eq:algo_DST_inverse}.
\end{lemma}

We will make use of the following result from \cite[Theorem~4]{PearsonWathen2012}.
\begin{lemma}[Pearson \& Wathen]\label{lem:pearson}
Let $M$ and $A$ be arbitrary symmetric positive definite matrices and let $\lambda \geq 0$ be an arbitrary nonnegative real number. Then, we have
  \[
\begin{aligned}
\frac{1}{2} \leq   \frac{ v^\top ( MA^{-1}M + \lambda A )\, v }{ v^{\top} ( M + \sqrt{\lambda} A ) A^{-1}  ( M + \sqrt{\lambda} A )\, v} \leq 1  &&&\forall\, v \in \Vh\setminus\{0\}.
\end{aligned}
  \]
\end{lemma}
 
\subsection{Proof of Theorem~\ref{thm:general_spectral_equivalence}}
Let $\tau$ and let $A$ be as in \eqref{eq:tauA_assumption}. As an intermediary step, we will use the bilinear form $s_{\dagger}(\cdot,\cdot) \approx s_D(\cdot,\cdot)$ defined on $\Vt$ by
\begin{equation}\label{eq:approx_D_def}
\begin{aligned}
s_{\dagger}(u,v) \coloneqq \sum_{n=1}^N\left\{  \int_{I_n} \frac{\tau_n}{\tau}  (\p_t \calI u, \p_t \calI v)_{M A^{-1} M}\,\dd t  +  \frac{1}{1+\delta_{nN}} \int_{I_n} \frac{\tau}{\tau_n} (u,v)_A \,\dd t \right\}.
\end{aligned}
\end{equation} 
The bilinear forms $s_D(\cdot,\cdot)$ and $s_{\dagger}(\cdot,\cdot)$ are spectrally equivalent, since \eqref{eq:tauA_assumption} and \eqref{eq:An_dual_spectral_equivalence} imply that $\frac{1}{ \alpha } s_{\dagger}(v,v) \leq s_D(v,v) \leq  \alpha  \,s_{\dagger}(v,v)$ for all $v\in\Vt$.
Therefore, these inequalities and Lemma~\ref{lem:spectral_SD} imply that $\frac{1}{\alpha} s_{\dagger}(v,v) \leq s(v,v) \leq 3 \, \alpha \, s_{\dagger}(v,v)$ for all $v\in\Vt$.
In matrix notation, the matrix $\bo S_{\dagger}$, that represents $s_{\dagger}(\cdot,\cdot)$ in the standard basis of $\Vt$, satisfies
\begin{equation}\label{eq:S_Dw_equiv}
\begin{aligned}
\frac{1}{ \alpha} \,\bo S_{\dagger} \leq \bo S \leq 3 \alpha \, \bo S_{\dagger}.
\end{aligned}
\end{equation}
Next we establish the connection between $\bo{S}_{\dagger}$ and $\bo H$ as follows. Let $u$ and $v \in \Vt$ be arbitrary; since $v$ and $\p_t \calI v$ are piecewise constant on each time-step, and since $\p_t \calI v|_{I_n} = \frac{1}{\tau_n}(v_n-v_{n-1})$ for each $n$, we can simplify~\eqref{eq:approx_D_def} to obtain
\[
s_{\dagger}(u,v) = \sum_{n=1}^N \left\{ \frac{1}{\tau} (u_n - u_{n-1},v_n-v_{n-1})_{M A^{-1} M} + \frac{1}{1+\delta_{nN}} \tau (u_n,v_n)_A \right\}.
\]
Therefore, Lemma~\ref{lem:orthogonality} implies that the change of basis to $\{\vphi_k\}_{k=1}^N$ gives
\begin{equation}\label{eq:sdag_transform}
\begin{aligned}
s_{\dagger}(u,v) = \frac{N}{2} \sum_{k=1}^N \left\{ \frac{\mu_k^2}{\tau} (\hat{u}_k,\hat{v}_k)_{M A^{-1} M} + \tau (\hat{u}_k, \hat{v}_k)_A  \right\} &&& \forall\,u,\, v\in \Vt,
\end{aligned}
\end{equation}
where $\{\hat{u}_k\}_{k=1}^N$ and $\{\hat{v}_k\}_{k=1}^N$ denote the coefficients of the basis expansion of $u$ and $v$ with respect to $\{\vphi_k\}_{k=1}^N$, as in \eqref{eq:algo_DST_inverse}.
In matrix notation, the identity~\eqref{eq:sdag_transform} shows that
\begin{equation}
  \bo S_{\dagger} = \bo\Phi^{\top} \bo{\hat{D}} \,\bo\Phi, \qquad
 \bo{\hat{D}} \coloneqq
\frac{N}{2} 
\Diag \left\{ \frac{\mu_{k}^2}{\tau} M A^{-1} M + \tau A \right\}_{k=1}^N,
\end{equation}
where we recall that $\bo\Phi$ is the matrix representation of the change of basis to $\{\vphi_k\}_{k=1}^N$.
In other words, the change of basis to $\{\vphi_k\}_{k=1}^N$ block-diagonalizes $\bo S_{\dagger}$.

To complete the proof, recall that $\bo H $ is defined by $\bo H = \bo \Phi^{\top} \bo{\hat{H}}\, \bo \Phi$, where $\bo{\hat{H}}$ is defined in \eqref{eq:algo_H_def}.
We therefore have the change of basis identities $\bo v^\top \bo S_{\dagger} \bo v = \bo{\hat{v}}^\top \bo{\hat{D}} \bo{\hat{v}}$ and $\bo v^\top \bo H \bo v = \bo{\hat{v}}^\top \bo{\hat{H}} \bo{\hat{v}}$ for any $\bo v\in \Vh^N$.
Hence, we can apply Lemma~\ref{lem:pearson} block-by-block to the matrices $\bo{\hat{D}}$ and $\bo{\hat{H}}$, and we deduce that 
\begin{equation}\label{eq:HD_equivalence}
\begin{aligned}
\frac{1}{2} \bo H \leq \bo S_{\dagger} \leq \bo H.
\end{aligned}
\end{equation}
Therefore, we obtain \eqref{eq:general_SH_equiv} from \eqref{eq:S_Dw_equiv} and \eqref{eq:HD_equivalence}.
Using the assumption~\eqref{eq:Hk_spectral_equivalence}, it is straightforward to show from the definition of $\bo H$ in \eqref{eq:algo_H_def} and $\widetilde{\bo H}$ in~\eqref{eq:Happrox_def} that $\gmin \widetilde{\bo H} \leq \bo H \leq \gmax \widetilde{\bo H}$.
Using these inequalities, we obtain the spectral equivalence~\eqref{eq:SHapprox_equiv} of $\bo S$ and $\widetilde{\bo H}$ from the equivalence~\eqref{eq:general_SH_equiv} between $\bo S$ and $\bo H$.
\qed

\begin{remark}[Weight terms]\label{rem:weight_term}
It might appear desirable to avoid the weight $(1+\delta_{nN})^{-1}$ in the bilinear forms $s_D(\cdot,\cdot)$ and $s_{\dagger}(\cdot,\cdot)$, which would have the advantage of tightening the constant in the upper bound~\eqref{eq:general_SH_equiv} from $3\alpha $  to $2\alpha$.
Moreover, the basis that block-diagonalizes the un-weighted version of $s_{\dagger}(\cdot,\cdot)$ is known explicitly through the functions $\{\phi_k\}_{k=1}^N$ where $\phi_k|_{I_n} \coloneqq \sin\left( \frac{(2k-1)\,n\,\pi}{2N +1}  \right)$ for each $n=1,\dots,N$.
However, the appearance of the term $2N+1$ in the denominators seems to be inconvenient for the implementation of the fast DST, which is why it is not considered further.
\end{remark}

\section{Numerical experiments}\label{sec:numexp}
We now study the efficiency, robustness and parallel scaling of the proposed method using a range of example problems in one, two, and three space dimensions.

\subsection{Condition numbers of Schur complement preconditioner}\label{sec:numexp1}
First, we assess the sharpness of the bounds in Theorem~\ref{thm:general_spectral_equivalence}, by computing numerically the extremal eigenvalues of the matrix $\bo H^{-1}\bo S$. For simplicity, we consider the one-dimensional heat equation on the spatial domain $\Omega=(0,1)$, with $T=1$, discretized by $P1$~FEM in space on a uniform mesh, and by the implicit Euler method with a uniform time-step size $\tau=T/N$. In this case, the matrices $M$ and $A_n=A$ represent respectively one dimensional mass and stiffness matrices.
Note that the reason for choosing in this experiment some low dimensional spatial problems is to guarantee the high accuracy of the eigenvalue solver. For this problem, the assumption~\eqref{eq:tauA_assumption} holds with $\alpha=1$. Therefore, the bound~\eqref{eq:general_SH_equiv} shows that $1/2\leq \lmin$ and $\lmax \leq 3$, where $\lmin$ and $\lmax$ denote respectively the minimal and maximal eigenvalues of $\bo H^{-1}\bo S$. 
This is in agreement with the results in Table~\ref{tab:eigenvalue_computations}, which suggest that the lower bound on the eigenvalues is indeed sharp, although the optimal upper bound appears rather to be $\lmax \leq 2$, leading to condition numbers $\kappa(\bo H^{-1}\bo S)\leq 4$ in these experiments.
Table~\ref{tab:eigenvalue_computations} leads to some further predictions. For instance, the results for $\lmax$ suggest that the damping parameter condition~\eqref{eq:damping_condition} simplifies here to $\omega< (1-\rhoa)/(1+\rhoa) \leq 1$.

\begin{table}[tb]
\begin{adjustbox}{max width=0.9\textwidth,center}
\begin{tabular}{| l | c | c | c | c | c | c | c| c| c |}
\hline
$h=1/64$ &  $N=4$ & $N=8$ & $N=16$ & $N=32$ & $N=64$ & $N=128$ & $N=256$ & $N=512$ & $N=1024$ \\ \hline
$\lmin$  & 0.8099 & 0.7080 & 0.6270 & 0.5728 & 0.5402 & 0.5223 & 0.5129 & 0.5081 & 0.5056 \\
$\lmax$ & 1.9999 & 1.9998 & 1.9996 & 1.9993 & 1.9986 & 1.9972 & 1.9944 & 1.9888 & 1.9780  \\
$\kappa(\bo H^{-1}\bo S)$ & 2.4693 & 2.8248 & 3.1893 & 3.4906 & 3.6994 & 3.8237 & 3.8885 & 3.9145 & 3.9122 \\ \hline 
 \multicolumn{10}{c}{~}  \\
\hline
$h=1/128$ &  $N=4$ & $N=8$ & $N=16$ & $N=32$ & $N=64$ & $N=128$ & $N=256$ & $N=512$ & $N=1024$ \\ \hline
$\lmin$ & 0.8099 & 0.7079 &  0.6270 & 0.5728 & 0.5402 &  0.5223 & 0.5129  & 0.5081 & 0.5056 \\
$\lmax$  & 2.0000 & 2.0000 &  1.9999 & 1.9998 & 1.9996 &  1.9993 & 1.9986 &   1.9972  & 1.9944 \\
$\kappa(\bo H^{-1}\bo S)$ & 2.4694 & 2.8250 &  3.1897 & 3.4916 & 3.7014 &  3.8278 & 3.8967 & 3.9310 & 3.9445\\ \hline      
\end{tabular}
\end{adjustbox}
\caption{Extremal eigenvalues and condition numbers of the matrix $\bo H^{-1}\bo S$ for the tests of Section~\ref{sec:numexp1}.}
\label{tab:eigenvalue_computations}
\end{table}

\subsection{Robustness with respect to time-steps, mesh-sizes, and approximate spatial solvers}\label{sec:numexp2}
We now study the robustness of the preconditioners with respect to variations in the time-steps, mesh-size, and also with respect to approximations in the spatial solvers that defined $\widetilde{\bo A}$ and $\widetilde{\bo H}$.
In this experiment, we consider the two dimensional heat equation on the spatial domain $\Om=(0,1)^2$, with $T=1$, and initial condition $u(0)=\sin(\pi x)\sin(\pi y)$. The problem is discretized in space by $P1$ FEM on a uniform mesh of sizes $h=2^{-k}$, $k=3,\dots,6$, and in time by the implicit Euler method with uniform time-steps. Again, for this problem, we also have $\alpha=1$.
The preconditioners $\widetilde{\bo A}$ and $\widetilde{\bo H}$ employ a fixed number of geometric multigrid (MG) V-cycles for the spatial solvers; i.e.\ the matrices $\widetilde{A}_n^{-1}$ and $\widetilde{H}_k^{-1}$ are defined as the application of either one or two MG V-cycles with Jacobi smoothers.
We then apply the inexact Uzawa method for these various choices of preconditioners, where we fix $\omega=0.9$ for the damping parameter.
In order to give a fair comparison of these different preconditioners, we compute at each iteration  $\norm{\bo u - \bo u_{j}}_{\bo S}$ for the exact error in the principal variable.
Figure~\ref{fig:varying_precond} shows the convergence histories of the inexact Uzawa method for these choices of spatial solvers, as well as for direct solvers. Here we fix $N=512$ and $h=1/64$. It is seen that the convergence rate does not depend significantly on the approximation of the spatial inverses.

\begin{figure}[tbh]
\begin{center}
\includegraphics{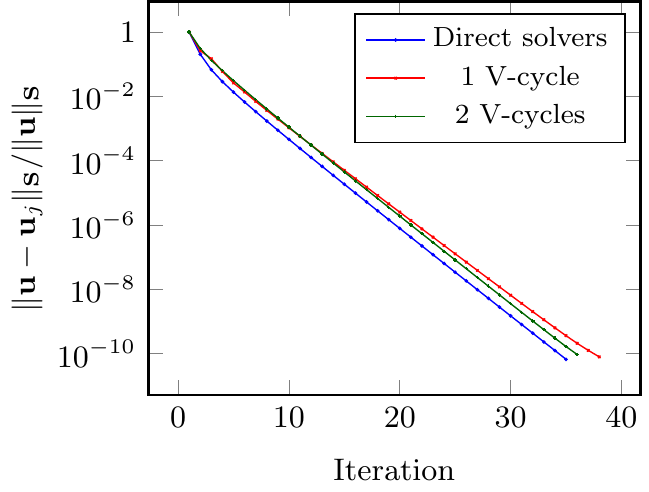}  
\end{center}
\caption{Convergence histories of the inexact Uzawa method for different choices of spatial solvers in Section~\ref{sec:numexp2}: direct solvers, one V-cycle and two V-cycles.}
\label{fig:varying_precond}
\end{figure}
\begin{table}[tbh]
\begin{adjustbox}{max width=0.9\textwidth,center}
\begin{tabular}{| l | c c c c |}
\hline
          & $h=1/8$ & $h=1/16$ & $h=1/32$ & $h=1/64$ \\ \hline
$N=128$ & 20 & 21 & 21 & 21 \\
$N=256$ & 21 & 22 & 22 & 22 \\
$N=512$ & 22 & 22 & 22 & 22  \\
$N=1024$ & 22 & 22 & 22 & 22 \\ \hline
\end{tabular}
\end{adjustbox}
\caption{Number of inexact Uzawa iterations required to satisfy $\norm{\bo u-\bo u_{j}}_{\bo S}<10^{-6} \norm{\bo u}_{\bo S}$ in the experiments of Section~\ref{sec:numexp2} (one MG V-cycle in space).}
\label{tab:meshtimestep_refinement}
\end{table}

Given the results in Figure~\ref{fig:varying_precond}, we now concentrate on the case of one V-cycle, and we now vary the mesh sizes and time-steps to verify the robustness with respect to the problem parameters.
The results are given in Table~\ref{tab:meshtimestep_refinement}, which presents the number of iterations required to achieve a relative error of $10^{-6}$ in the $\bo S$-norm of the error for the principal variable.
It is thus seen that the resulting method is robust with respect to the mesh size and number of time-steps. It is also robust with respect to the ratio of mesh and time-step sizes, as expected.

\subsection{Time-parallel computations}\label{sec:numexp3}
We now present our main set of numerical experiments, involving weak and strong scaling studies with large-scale parallel computations.
All the parallel computations in this work were performed on the Vulcan BlueGene/Q Supercomputer in Livermore, California.
We start with weak and strong scaling tests of the time-parallelism.
We solve the three dimensional heat equation on the spatial domain $\Omega = (0,1)^3$, with $T = 0.1$. The right hand side and the initial condition are chosen such that the exact solution is $u(t,x,y,z) = e^{-3 \pi^2 t} \sin(\pi x) \sin(\pi y) \sin(\pi z)$. The spatial discretization uses lowest-order hexahedral finite elements on a fixed uniform mesh with $4096$ elements. Uniform time-steps are used in time. 
For the weak scaling test, we assign $16$ time steps to each core; each core then has a fixed problem size of $157\,216$ unknowns.
For the strong scaling test we fix the number of time steps $N=65\,536$, which results overall in a global linear system of dimension $643\,956\,736$.
For the spatial solvers, we use an algebraic multigrid method (AMG) provided by the library \emph{hypre}~\cite{hypre}, and the parallel DST (based on the FFT) is provided by the library FFTW3~\cite{FFTW05}. We further use GMRES as an acceleration method for the inexact Uzawa method.

\begin{table}[tbh]
\begin{adjustbox}{max width=0.9\textwidth,center}
\begin{tabular}{| r | r | r | c | c c | c | c |}
\hline 
procs & $N$ & dofs & iter & time/iter & total time & time FFT (\%) & time AMG (\%) \\ \hline
1 & 16  & 157 216 & 15 & 1.87  &  28.00 &  0.9\%  &  84.5\% \\
2 & 32  & 314 432 & 15 & 1.85  &  27.75 &  1.5\%  &  83.4\% \\
4 & 64  & 628 864 & 15 & 1.81  &  27.16 &  1.7\%  &  82.8\% \\
8 & 128  &  1 257 728  & 15  & 1.77 & 26.60 & 1.9\% & 82.4\% \\
16 & 256  &  2 515 456  & 15  & 1.78 & 26.72 & 2.1\% & 82.1\% \\
32 & 512  &  5 030 912  & 15  & 1.79 & 26.78 & 2.3\% & 82.0\% \\
64 & 1 024 &  10 061 824 & 16  & 1.79 & 28.66 & 3.0\% & 81.3\% \\
128 & 2 048 & 20 123 648 & 19 & 1.81 & 34.35 & 4.1\% & 79.8\% \\
256 & 4 096 & 40 247 296 & 20 & 1.81 & 36.11 & 4.2\% & 79.5\% \\
512 & 8 192 & 80 494 592 & 21 & 1.80 & 37.88 & 4.2\% & 79.3\% \\
1 024 & 16 384 & 160 989 184 & 22 & 1.81 & 39.77 & 4.4\% & 79.0\% \\
2 048 & 32 768 & 321 978 368 & 22 & 1.82 & 40.10 & 5.3\% & 78.3\% \\
4 096 & 65 536 & 643 956 736 & 22 & 1.87 & 41.09 & 7.4\% & 76.4\% \\ \hline
\end{tabular}
\end{adjustbox}
\caption{Weak scaling test of Section~\ref{sec:numexp3}. Computational times in seconds.}
\label{tab:weakstrongscaling}
%\end{table}
%\vspace{-1cm}
%\begin{table}[tbh]
\vspace{1ex}
\begin{adjustbox}{max width=0.9\textwidth,center}
\begin{tabular}{| r | c | c | c | r r | c | c |}
\hline
procs & $N$ & dofs & iter & time/iter &  total time &  time FFT (\%) & time AMG (\%) \\ \hline
16   & 65 536 & 643 956 736 & 22 & 310.18 &  6823.88 & 3.9\% & 72.9\% \\
32   & 65 536 & 643 956 736 & 22 & 155.68 &  3425.04 & 4.1\% & 72.9\% \\
64   & 65 536 & 643 956 736 & 22 & 78.66  &  1730.53 & 4.8\% & 72.4\% \\
128  & 65 536 & 643 956 736 & 22 & 39.98  &  879.52  & 5.5\% & 72.0\% \\
256  & 65 536 & 643 956 736 & 22 & 20.89  &  459.60  & 7.1\% & 70.5\% \\
512  & 65 536 & 643 956 736 & 22 & 10.76  &  236.82  & 7.3\% & 70.9\% \\
1024 & 65 536 & 643 956 736 & 22 & 5.65  &  124.22 & 6.8\% & 72.3\% \\
2048 & 65 536 & 643 956 736 & 22 & 3.13  &  68.79  & 7.0\% & 74.1\% \\
4096 & 65 536 & 643 956 736 & 22 & 1.87  &  41.09  & 7.4\% & 76.4\% \\ \hline
\end{tabular}
\end{adjustbox}
\caption{Strong scaling test of Section~\ref{sec:numexp3}. Computational times in seconds.}
\label{tab:strongstrongscaling}
\end{table}

Tables~\ref{tab:weakstrongscaling} and~\ref{tab:strongstrongscaling} report the results of the weak and strong scaling tests, including the total number of iterations to reach a residual tolerance of $10^{-8}$, the time spent per iteration and the total time of the computation. The last two columns in each table show the percentages of the total time that are spent inside calls to the parallel FFT and the spatial solvers.
In the weak scaling test, we see that for small values of $N$, the number of iterations initially increases up to a maximum of $22$, after which it remains constant; this is explained by the pre-asymptotic behaviour of the condition numbers reported in Table~\ref{tab:eigenvalue_computations} above.
Most importantly, the time per iteration remains essentially constant, so the small increase in total time is due to the iteration count. In this sense the test shows very good weak scaling of the method.
Table~\ref{tab:strongstrongscaling} shows the good strong scaling of the method, where a doubling of the number of processors reduces the computational time by a factor often close to two.
The last two columns of both Tables~\ref{tab:weakstrongscaling} and \ref{tab:strongstrongscaling} show that the use of the FFT amounts to only a small part of the total computational time, whereas the AMG calls are dominant, with some additional time being spent outside of both FFT and AMG calls, for instance to handle some of the vector operations outside the preconditioners.

\subsection{Space-time parallel computations}\label{sec:numexp4}

We now consider space-time parallelism, where additional cores are used to apply the spatial solvers in parallel over space. The spatial parallelism here is provided by the library \emph{hypre}. 
We consider the same right hand side, initial condition and space-time domain as in the previous example. For the spatial approximations we again use lowest order hexahedral elements, where we decompose the spatial domain $\Omega = (0,1)^3$ into $262\,144$ elements, and we use 4096 time steps for the time discretization. This results in a global linear system for $2\,249\,728\, 000$ unknowns. With the same solver settings as in Section~\ref{sec:numexp3}, we use varying numbers of processors in space and time, denoted respectively by $p_x$ and $p_t$, thereby resulting in a total of $p_x p_t$ processors.

\begin{table}[tbh]
\begin{adjustbox}{max width=0.9\textwidth,center}
\begin{tabular}{|cr|rrrrrr|}
\hline
& & \multicolumn{6}{c|}{procs w.r.t. space $p_x$} \\
& & 16 & 32 & 64 & 128 & 256 & 512   \\ \hline
\multirow{13}{1em}{\begin{sideways} procs w.r.t. time $p_t$\end{sideways}}
& 4      & 12\,158.70 & 7\,000.47 & 4\,381.72 & 2\,925.62 & 2\,132.41 & 2\,107.73 \\
& 8      & 6\,721.02  & 3\,911.30 & 2\,437.63 & 1\,654.01 & 1\,219.39 & 1\,170.38 \\
& 16     & 4\,016.91  & 3\,522.05 & 1\,459.71 & 1\,007.60 & 728.52    & 703.79    \\
& 32     & 2\,203.77  & 1\,946.12 & 822.15    & 565.93    & 421.31    & 418.68    \\
& 64     & 1\,212.84  & 9\,04.27  & 429.03    & 304.47    & 238.31    & 245.17    \\
& 128    & 667.20     & 468.11    & 220.43    & 162.00    & 130.97    & 135.74    \\
& 256    & 341.14     & 232.08    & 117.75    & 85.76     & 70.97     & 74.36     \\
& 512    & 172.21     & 119.18    & 59.54     & 44.76     & 37.58     &           \\
& 1\,024 & 84.94      & 60.44     & 30.12     & 23.07     &           &           \\
& 2\,048 & 44.92      & 31.73     & 15.96     &           &           &           \\
& 4\,096 & 27.94      & 21.29     &           &           &           &	          \\	
\hline
\end{tabular}
\end{adjustbox}
\caption{Space-time parallel test of Section~\ref{sec:numexp4}. Computational times in seconds. The total number of processors is $p_tp_x$, with up to $131\,072$ processors in total. The linear system involves $2\,249\,728\, 000$ unknowns.}
\label{tab:spacetimescaling}
\end{table}

Table \ref{tab:spacetimescaling} presents the total computational times in seconds for these different processor configurations.
We observe almost perfect scaling with respect to the time-parallelization, whereas the space-parallelization stagnates beyond $128$ processors. The best result was obtained for $p_x = 64$ and $p_t = 2048$, i.e.\ with $131\,072$ processors overall, resulting in a time to solution of $15.96$ seconds.
Concerning the spatial parallelism, we used the default library settings in \emph{hypre} without further tuning. 
A possible reason for the observed behaviour in terms of the spatial parallelism is that the spatial solvers adapt the smoothers inside AMG to the number of cores, with more effective sequential smoothers available on lower core counts. This leads to an additional reduction of the total number of iterations for small $p_x$.

\section*{Conclusion}
We have presented an original method for the time-parallel solution of parabolic problems.
The inf-sup theory of the discrete problem provided the motivation lead to the derivation of an original symmetric saddle-point reformulation of the problem that remains stable with respect to the natural norms of the problem.
The saddle-point system can then be solved efficiently by iterative methods, such as the inexact Uzawa method considered here.
We proposed an easily implementable non-intrusive time-parallel preconditioner for the Schur complement of the system, and proved robust spectral bounds with respect to key discretization parameters.
The robustness, efficiency and parallel performance of the proposed method were shown both theoretically and experimentally in large scale parallel computations.

More broadly, the approach for preconditioning nonsymmetric systems as pursued in this work, namely the construction of preconditioners based on the inf-sup stability of the problem, should not be strictly limited to the parabolic PDEs considered here, since inf-sup stability is equivalent to well-posedness for general linear operators. However, it is natural to expect that the practical details for developing efficient solvers will be specific to each problem and their discretizations.

\begin{small}

\end{small}

\end{document}